\documentclass[a4paper,reqno]{amsart}
\usepackage{amssymb, amsmath, amscd}
\def\SR{S_{\mathcal{R}}}
\def\DR{{\mathcal{D}/\mathcal{R}}}
\def\DN{{\mathcal{D}/\mathcal{N}}}

\def\BB{{\mathcal{B}}}\def\RR{{\mathcal{R}}}
\def\DD{{\mathcal{D}}}
\def\Xone{\varepsilon_{\RR,\alpha_1-1,\alpha_2}}
\def\Xtwo{-\textstyle\frac12\{\frac{\alpha_1}{\alpha_1-1}\varepsilon_{\mathcal{R},\alpha_1-1,\alpha_2}
     +\frac{\alpha_2}{\alpha_2-1}\varepsilon_{\RR,\alpha_1,\alpha_2-1}\}}
\def\Xthree{\varepsilon_{\RR,\alpha_1,\alpha_2-1}}
\def\Xfour{\varepsilon_{\RR,\alpha_1-2,\alpha_2}}
\def\Xfive{-\textstyle\frac12\{\frac{\alpha_1-1}{\alpha_1-2}\varepsilon_{\mathcal{R},\alpha_1-2,\alpha_2}
     +\frac{\alpha_2}{\alpha_2-1}\varepsilon_{\RR,\alpha_1-1,\alpha_2-1}\}}
\def\Xsix{\varepsilon_{\RR,\alpha_1,\alpha_2}}
\def\Xseven{\varepsilon_{\RR,\alpha_1,\alpha_2-2}}
\def\Xeight{-\textstyle\frac12\{\frac{\alpha_1}{\alpha_1-1}\varepsilon_{\mathcal{R},\alpha_1-1,\alpha_2-1}
     +\frac{\alpha_2-1}{\alpha_2-2}\varepsilon_{\RR,\alpha_1,\alpha_2-2}\}}
\def\Xnine{-\frac12\frac{\alpha_1^2-2\alpha_1+\alpha_2^2-2\alpha_2+1}{3-\alpha_1-\alpha_2}
\varepsilon_{\RR,\alpha_1,\alpha_2}}
\def\Xten{\frac{\alpha_1^2+\alpha_2^2{-1}}{4(3-\alpha_1-\alpha_2)}\varepsilon_{\RR,\alpha_1,\alpha_2}
+\frac14\frac{\alpha_1\alpha_2}{(\alpha_1-1)(\alpha_2-1)}\varepsilon_{\RR,\alpha_1-1,\alpha_2-1}}

\def\Xtwelve{-\varepsilon_{\RR,\alpha_1,\alpha_2}}
\def\Xthirteen{0}
\def\Xfourteen{\varepsilon_{\RR,\alpha_1-1,\alpha_2-1}}
\def\Xfifteen{\{-\textstyle\frac{1}{\alpha_1-1}\varepsilon_{\RR,\alpha_1-1,\alpha_2}
     -\frac{1}{\alpha_2-1}\varepsilon_{\RR,\alpha_1,\alpha_2-1}\}}
\def\Xsixteen{\frac{1}{(\alpha_1-1)(\alpha_2-1)}\varepsilon_{\RR,\alpha_1-1,\alpha_2-1}
+\frac2{3-\alpha_1-\alpha_2}\varepsilon_{\RR,\alpha_1,\alpha_2}}
\def\Xseventeen{\{-\textstyle\frac{1}{\alpha_1-2}\varepsilon_{\RR,\alpha_1-2,\alpha_2}
     -\frac{1}{\alpha_2-1}\varepsilon_{\RR,\alpha_1-1,\alpha_2-1}\}}
\def\Xeighteen{\{-\textstyle\frac{1}{\alpha_1-1}\varepsilon_{\RR,\alpha_1-1,\alpha_2-1}
     -\frac{1}{\alpha_2-2}\varepsilon_{\RR,\alpha_1,\alpha_2-2}\}}
\def\Xnineteen{\frac{\alpha_1+\alpha_2}{3-\alpha_1-\alpha_2}\varepsilon_{\RR,\alpha_1,\alpha_2}
     +\frac12\frac{\alpha_1+\alpha_2}{(\alpha_1-1)(\alpha_2-1)}\varepsilon_{\RR,\alpha_1-1,\alpha_2-1}}
\newtheorem{theorem}{Theorem}[section]

\newtheorem{lemma}[theorem]{Lemma}

\newtheorem{remark}[theorem]{Remark}

\newtheorem{conjecture}[theorem]{Conjecture}

\makeatletter
 \@addtoreset{equation}{section}
\makeatother
\begin{document}
\title[Heat Content Asymptotics]
{Neumann heat content asymptotics with singular initial temperature and singular specific heat}
\author{M. van den Berg}
\address
{School of Mathematics, University of Bristol\\
University Walk, Bristol BS8 1TW\\United Kingdom}
\begin{email}{M.vandenBerg@bris.ac.uk}\end{email}
\author{P. Gilkey}
\address{Mathematics Department, University of Oregon\\
Eugene, OR 97403\\USA}
\begin{email}{gilkey@uoregon.edu}\end{email}
\author{H. Kang}
\address{Department of Mathematics, Korea Institute for Advanced Study,
Hoegiro 87,\smallbreak\noindent Dongdaemun-gu, Seoul 130-722, Korea}
\begin{email}{h.kang@kias.re.kr}\end{email}
\begin{abstract} We study the asymptotic behavior of the
heat content on a compact Riemannian manifold with boundary and
with singular specific heat and singular initial temperature
distributions imposing Robin boundary conditions. Assuming the existence of a complete asymptotic
series we determine the first three terms in that series. In
addition to the general setting, the interval is studied in
detail as are recursion relations among the coefficients and the relationship between
the Dirichlet and Robin settings.
 \end{abstract}

\maketitle
\section{Introduction}
\subsection{Historical framework} Let $M$ be a compact $m$-dimensional manifold
with boundary $\partial M$ and Riemannian metric $g$. Let $D$ be an operator
of Laplace type which drives the evolution process (see Section~\ref{sect-1.2} for details).
Let $\phi$ represent the initial
temperature of $M$ and let $T(x;t)$ be subsequent temperature of the manifold.
A good general reference to the heat equation is provided by
\cite{Ca45,G95}. We impose suitable boundary conditions $\BB$ to ensure the process is well defined. Let
$\rho$ be the specific heat. We suppose for the moment $\phi$ and $\rho$ are smooth
and let $\beta(t)$ be the heat content of the manifold.
There is a complete asymptotic series for $\beta$ as $t\downarrow0$ with locally computable
coefficients. We postpone for the moment a precise definition to avoid complicating the
exposition at this stage.

The problem was first studied by \cite{BD89,BlG94} in the context of domains in Euclidean space
with smooth boundaries. We note that regions with polygonal boundaries were examined in
 \cite{BeSr90}.  Subsequent results for regions with fractal boundaries were obtained in \cite{vdB,vdBdH}.
Apart from these papers, most of the work on the
heat content asymptotics has been in the smooth category and we shall always assume the
boundary of $M$ to be smooth.
The ``functorial method'' has been used to express the asymptotic coefficients in terms of geometrical
quantities for Dirichlet and Robin boundary conditions
by \cite{BeDeGi93,BG94,BG96,BeGi99,DS94}. Other authors
\cite{McA93,McMe03a,PhJa90,Sav98a,SAV01} have examined these coefficients using
other methods. Various other boundary conditions also have been considered
\cite{BGKK07,GiKixy,GiKiPa02} as have variable geometries \cite{G99}.
Motivated by work of \cite{x5} for the growth of the heat trace asymptotics, the growth
of the heat content asymptotics has been examined \cite{BGK11x,GM12,TS11}.

The case where the initial temperature and the specific heat are singular in a controlled
fashion at the boundary will form the centerpiece of this paper -- we refer to Section~\ref{sect-1.3}
for precise definitions. A rigorous treatment was
provided in \cite{BGS} where the initial temperature is singular on the boundary but the
specific heat is smooth. Subsequently, the case where the specific heat can be singular
as well was studied for Dirichlet boundary conditions \cite{BG12,BGKK07}. The present
paper examines the situation for Neumann boundary conditions.

\subsection{Operators of Laplace type}\label{sect-1.2}
Let $(M,g)$ be a compact Riemannian manifold of dimension $m$
with smooth boundary $\partial M$. Let $V$ be a smooth vector bundle
over $M$ and let $D:C^\infty(V)\rightarrow C^\infty(V)$ be a second order partial
differential operator on the space of smooth sections to $V$.
We adopt the {\it Einstein convention} and
sum over repeated indices. We say that $D$ is of {\it Laplace type} if
the leading symbol of $D$ is given by the metric tensor, i.e.
locally $D$ has the form:
$$\displaystyle D=-\left\{ g^{\mu\nu}\partial_{x_\mu}
\partial_{x_\nu}+A^\sigma\partial_{x_\sigma}+B\right\}\,.
$$
The following {\it Bochner formalism} \cite{G95} permits
us to work tensorially:
\begin{lemma}\label{lem-1.1} There exists a unique
connection $\nabla$ on $V$ and a unique endomorphism $E$ of $V$ so
that
$D=D(g,\nabla,E)=-(g^{\mu\nu}\nabla_{\partial_{x_\nu}}\nabla_{\partial_{x_\mu}}+E)$.
The connection $1$-form $\omega$ of
$\nabla$ and the endomorphism $E$ are given by
$$
\omega_\mu=\textstyle{\textstyle\frac12}(g_{\mu\nu}A^\nu
+g^{\sigma\varepsilon}\Gamma_{\sigma\varepsilon\mu}\operatorname{Id})\text{ and }
E=B-g^{\mu\nu}(\partial_{x_\nu}\omega_{\mu}+\omega_{\mu}\omega_{\nu}
    -\omega_{\sigma}\Gamma_{\mu\nu}{}^\sigma)\,.
$$\end{lemma}

If $D=d^*d+dd^*$ is the Laplacian on the space of smooth $p$-forms,
then $\nabla$ is the Levi-Civita connection and $E$ is given in terms of curvature. If $D$
is the spin Laplacian, then $\nabla$ is the spinor connection and $E$ is a multiple of the
scalar curvature. We shall use the Levi-Civita connection and $\nabla$
to covariantly differentiate tensors of all types.

We must impose suitable boundary conditions. Let $e_m$ be the inward unit
normal vector field on the boundary.
Let $\SR$ be an auxiliary endomorphism of $V|_{\partial M}$ and let
$\nabla$ be the connection on $V$ given in Lemma~\ref{lem-1.1}.
Let $\BB_\DD$ and $\BB_\RR$ be the
{\it Dirichlet} and the {\it Robin} boundary operators which are defined, respectively,
by setting:
$$
\BB_\DD \phi=\phi|_{\partial M}\quad\text{and}\quad
\BB_\RR\phi=(\nabla_{e_m}\phi+\SR \phi)|_{\partial M}\quad\text{for}\quad
f\in C^\infty(V)\,.
$$
The {\it Neumann} boundary operator is defined by taking $\SR=0$. We let $D_\DR$
denote either the Dirichlet or the Robin realization of $D$ henceforth. It is convenient
to have a common notation.

\subsection{The initial temperature and the specific heat}\label{sect-1.3}
 Let $r$ be the geodesic distance to the boundary. This is a continuous
function on all of $M$ which is smooth near $\partial M$. If
$\vec y=(y^1,...,y^{m-1})$ is a system of local coordinates near a point
$P\in\partial M$, then
$\vec x=(y,r)$ for $r\in[0,\epsilon]$ is an {\it adapted} system of local coordinates near $P$ in $M$
for some $\epsilon>0$;
the curves $\gamma_y(s):=(y,s)$ are unit speed geodesics perpendicular
to the boundary and $\partial_r$ is the inward pointing unit geodesic normal.

To have a common notation, let $W=V$ or let $W=V^*$. If $\alpha\in\mathbb{C}$,
let $\mathcal{K}_\alpha(W)$ be the vector space of all
sections $\Phi$ to $W$ which are smooth on the interior of $M$ such that
$r^\alpha\Phi$ is smooth near $\partial M$. The parameter $\alpha$ controls
the blow up (resp. decay) of $\Phi$ near the boundary if $\Re(\alpha)>0$ (resp.
$\Re(\alpha)<0$).
Let
$\phi\in\mathcal{K}_{\alpha_1}(V)$ represent the initial temperature of
$M$ and let $\rho\in\mathcal{K}_{\alpha_2}(V^*)$ give the specific
heat. We shall always assume that
\begin{equation}\label{eqn-1.a}
\Re(\alpha_1)<1,\quad \Re(\alpha_2)<1,\quad\text{and}\quad
\alpha_1+\alpha_2\notin\mathbb{Z}\,.
\end{equation}
The first two conditions ensure that $\phi$ and $\rho$ are in $L^1$.  The final
condition is imposed to avoid interactions between the interior and boundary terms
in the asymptotic series for the heat content. We will return to this point subsequently.

Expand $\phi\in\mathcal{K}_{\alpha_1}(V)$ near $\partial M$ in a modified Taylor series:
\begin{equation}\label{eqn-1.b}
\phi(y,r)\sim r^{-\alpha_1}\sum_{j=0}^\infty\phi^j(y)r^j
\quad\text{as}\quad r\downarrow0^+\text{ where }
\phi^j=\textstyle\frac1{j!}(\nabla_{\partial_r})^j(r^{\alpha_1}\phi)|_{\partial M}\,.
\end{equation}
We shall usually be working with scalar
operators and can choose a local section $s$ so that
$\nabla_{e_m}s=0$. We may then regard $\phi$ as a function and the above expansion as a
Taylor series. A similar expansion, of course, is valid for the specific heat $\rho$
where we regard $\rho$ as a section to the dual bundle and covariantly
differentiate with respect to
the dual connection on $V^*$ which has connection $1$-form $-\omega_\nu^\star$.

\subsection{The heat equation and the heat content}
If $\phi\in\mathcal{K}_{\alpha_1}(V)$, then the subsequent temperature
$T:=e^{-tD_{\BB}}\phi$ is characterized by the relations:
\begin{equation}\label{eqn-1.c}
\begin{array}{lll}
(\partial_t+D)T=0&\text{(evolution equation)},\\
\lim_{t\downarrow 0}T(\cdot;t)=\phi&\text{(initial condition)},
\vphantom{\vrule height 12pt}\\
\BB T(\cdot;t)=0\text{ for }t>0&\text{(boundary condition)}.
\vphantom{\vrule height 12pt}
\end{array}\end{equation}
Let $\langle\cdot,\cdot\rangle$ be the {\it natural pairing} between $V$ and
the dual bundle $V^*$, let $dx$ be the {\it Riemannian
measure} on $M$, let $dy$ be the Riemannian measure on $\partial M$, and let $\rho\in\mathcal{K}_{\alpha_2}(V^*)$ be
the {\it specific heat} of the manifold. The total {\it heat content} of the manifold is defined by setting:
$$
\beta(\phi,\rho,D,\BB)(t):=\int_M \langle
e^{-tD_{\BB}}\phi,\rho\rangle dx\,.
$$
There is a smooth {\it heat kernel} $K=K_{D,\BB}$ so that
$T(x;t)=\int_MK(x,\tilde x;t)\phi(\tilde x)d\tilde x$, and
$$
\beta(\phi,\rho,D,\mathcal{B})(t)=\int_{M\times M}\langle K(x,\tilde x;t)
\phi(\tilde x),\rho(x)\rangle d\tilde xdx\,.
$$
This is well defined for $\phi\in L^1(V)$ and $\rho\in L^1(V^*)$;
it was for this reason that we assumed $\Re(\alpha_1)<1$
and $\Re(\alpha_2)<1$ in Equation~(\ref{eqn-1.a}).

If $D_{\BB}$ is self-adjoint, as will be the case for
either the Dirichlet or the Robin realization of the scalar Laplacian, then
we can take a {\it spectral resolution} $\{\phi_n,\lambda_n\}$ for $D_{\BB}$.
Here the $\{\phi_n\}$ is an orthonormal basis for $L^2(V)$ such that
$D\phi_n=\lambda_n\phi_n$ and $\BB\phi_n=0$. We then have
$$K(x,\tilde x;t)=\sum_ne^{-t\lambda_n}\phi_n(x)\otimes\phi_n(\tilde x)\,.$$
This series converges in the $C^\infty$ topology for $t>0$.
The temperature $T$ is smooth for $t>0$. However, the
 convergence to $\phi$ as $t\downarrow0$ in Equation~(\ref{eqn-1.c}) is not pointwise but in 
 $L^1(V)$.

The following example is instructive. Let $M=[0,\pi]$ and let $D=-\partial_x^2$.
The spectral resolution of the Dirichlet Laplacian is given by
$\{n^2,\sqrt{2/\pi}\sin(nx)\}_{n\ge1}$.
Thus if $\phi=1$, we have
$$T(x;t)=\frac2{\pi}\sum_{n=1,n\text{ odd}}^\infty \frac2n e^{-tn^2}\sin(nx)\,.$$
This is smooth for $t>0$. However, since
$T(0;t)=0$ for $t>0$ and $\phi=1$, the convergence is not pointwise. The
associated heat content is given by:
\begin{eqnarray*}
\beta(1,1,\Delta,\BB_\DD)(t)&=&
\frac8\pi\sum_{n=1,n\text{ odd}}^\infty\frac1{n^2}e^{-tn^2}\\
&=&\pi-\frac4{\sqrt\pi}t^{1/2}+O(t^k)\text{ as }t\downarrow0\ \forall k\ge1\,.
\end{eqnarray*}

\subsection{The form of the asymptotic series}\label{sect-1.5}
We shall need to consider certain integrals which are divergent and which need to be
{\it regularized}.  For example, the integral
$\int_M\langle\phi,\rho\rangle dx$ is divergent if $1<\Re(\alpha_1+\alpha_2)<2$.
The Riemannian measure is not in general a product near the boundary.
Since, however, $dx=dydr$ on the boundary of
$M$, we may decompose
$$\langle\phi,\rho\rangle dx=\langle\phi^0,\rho^0\rangle r^{-\alpha_1-\alpha_2}dydr
+O(r^{1-\alpha_1-\alpha_2})\,.$$
Let $C_\varepsilon:=\{x\in M:r(x)\le\varepsilon\}$
be a small collared neighborhood of
the boundary. For $\Re(\alpha_1+\alpha_2)<2$ and $\alpha_1+\alpha_2\ne1$, define:
\begin{eqnarray*}
\mathcal{I}_{Reg}^g(\phi,\rho)
&:=&\displaystyle\phantom{+}\int_{M-{\mathcal{C}_\epsilon}}\langle\phi,\rho\rangle dx
+\int_{\mathcal{C}_\epsilon}
\left\{\langle\phi,\rho\rangle dx-\langle\phi^0,\rho^0\rangle r^{-\alpha_1-\alpha_2}dydr\right\}\\
&&\displaystyle+\int_{\partial M}\langle\phi^0,\rho^0\rangle dy
\times\varepsilon^{1-\alpha_1-\alpha_2}(1-\alpha_1-\alpha_2)^{-1}\,.
\end{eqnarray*}
This is clearly independent of $\varepsilon$
and agrees with $\int_M\langle\phi,\rho\rangle dx$ if $\Re(\alpha_1+\alpha_2)<1$.
The regularization ${\mathcal{I}_{\operatorname{Reg}}}(\phi,\rho)$ is a meromorphic
function of $\alpha_1+\alpha_2$ with a simple pole at
$\alpha_1+\alpha_2 = 1$.
More generally, the integrals $\langle D^n\phi,\rho\rangle$ which appear in Conjecture~\ref{conj-1.2}
below need regularization if $\Re(\alpha_1+\alpha_2)>1-2n.$ Poles can appear
whenever $\Re(\alpha_1+\alpha_2)=1-k$. These poles are evident in the
formulas given subsequently; we expect the local formulas for the
asymptotic expansion of the heat content may involve log terms
when $\alpha_1+\alpha_2\in\mathbb{Z}$ and
for that reason excluded these values in Equation~(\ref{eqn-1.a}).

To simplify the notation, we set:
$$\beta_n^M(\phi,\rho,D):=(-1)^n/n!\cdot\mathcal{I}_{Reg}^g\{\langle D^n\phi,\rho\rangle\}\,.$$
If $M$ is a closed manifold, these are the invariants which would appear in the heat content
expansion. We assume that the following conjecture
(which extends the discussion of  \cite{BG12}) holds henceforth. We will justify the
powers $t^{(1+j-\alpha_1-\alpha_2)/2}$ subsequently in Section~\ref{sect-4} using dimensional analysis.
We refer to \cite{BGS}
where a related result was established when the specific heat is smooth.

\begin{conjecture}\label{conj-1.2}
If $(\alpha_1,\alpha_2)$ satisfy Equation~(\ref{eqn-1.a}), then
there is a complete asymptotic series as $t\downarrow0$ of the form:
\begin{eqnarray*}
\beta(\phi,\rho,D,\BB_\DR)(t)&\sim&
\sum_{n=0}^\infty t^n\beta_n^M(\phi,\rho,D)\\
&+&\sum_{j=0}^\infty t^{(1+j-\alpha_1-\alpha_2)/2}
   \beta_{j,\alpha_1,\alpha_2}^{\partial
   M}(\phi,\rho,D,\BB_\DR)\,.
\end{eqnarray*}
The coefficients $\beta_{j,\alpha_1,\alpha_2}^{\partial M}(\phi,\rho,D,\BB_\DR)$ are given by
integrals of local invariants over the boundary.
\end{conjecture}

\begin{remark}\rm This conjecture has been established in \cite{BGS} using the
calculus of pseudo-differential operators in the special case
that $\alpha_2\in\mathbb{N}$ or that $\alpha_1\in\mathbb{N}$. The extension to
the present setting is motivated by that work.
\end{remark}

Let $R_{ijkl}$ denote the Riemann curvature tensor; with our sign convention, we have that
 $R_{1221}=+1$
on the unit sphere $S^2$ in $\mathbb{R}^3$. Let $\operatorname{Ric}$ denote the Ricci
tensor, let $\tau$ denote the scalar curvature,
and let $L_{ab}$ denote the second fundamental form. We let indices $\{i,j,k,l\}$ range
from $1$ to $m$ and index a local orthonormal frame for $TM$; we let indices $\{a,b,c\}$
range from $1$ to $m-1$ and index a local orthonormal frame for $T\partial M$. On
the boundary, $e_m$ will always denote the inward unit geodesic normal and `;' will
denote the components of the covariant derivative.
In Section~\ref{sect-2}, we will perform a careful analysis of the local formulas involved
and show:\begin{lemma}\label{lem-1.4}
There exist universal constants
$\varepsilon_{\DR,\alpha_1,\alpha_2}^\nu$ so that:
\medbreak
$\beta_{0,\alpha_1,\alpha_2}^{\partial M}
     (\phi,\rho,D,\BB_\DR)=\displaystyle\int_{\partial M}
       \varepsilon_{\DR,\alpha_1,\alpha_2}
\langle\phi^0,\rho^0\rangle dy$.
\medbreak$
\beta_{1,\alpha_1,\alpha_2}^{\partial M}
(\phi,\rho,D,\BB_\DR)
=\displaystyle\int_{\partial M}\left\{\varepsilon_{\DR,\alpha_1,\alpha_2}^1
\langle\phi^1,\rho^0\rangle
+\varepsilon_{\DR,\alpha_1,\alpha_2}^2
\langle L_{aa}\phi^0,\rho^0\rangle\right.$
\medbreak$\left.\quad
+\varepsilon_{\DR,\alpha_1,\alpha_2}^3
\langle\phi^0,\rho^1\rangle
+\varepsilon_{\RR,\alpha_1,\alpha_2}^{15}
\langle \SR\phi^0,\rho^0\rangle\right\}dy$.
\medbreak$
\beta_{2,\alpha_1,\alpha_2}^{\partial M}
(\phi,\rho,D,\BB_\DR)
    =\displaystyle\int_{\partial M} \left\{\varepsilon_{\DR,\alpha_1,\alpha_2}^4
    \langle\phi^2,\rho^0\rangle
+\varepsilon_{\DR,\alpha_1,\alpha_2}^5\langle
L_{aa}\phi^1,\rho^0\rangle\right.$
\medbreak$\quad
    +\varepsilon_{\DR,\alpha_1,\alpha_2}^6\langle E
    \phi^0,\rho^0\rangle$
$+\varepsilon_{\DR,\alpha_1,\alpha_2}^7\langle
\phi^0,\rho^2\rangle
   +\varepsilon_{\DR,\alpha_1,\alpha_2}^8\langle
   L_{aa}\phi^0,\rho^1\rangle$
\medbreak$\quad
    +\varepsilon_{\DR,\alpha_1,\alpha_2}^9\langle
    \operatorname{Ric}_{mm}\phi^0,\rho^0\rangle
+\varepsilon_{\DR,\alpha_1,\alpha_2}^{10}\langle
 L_{aa}L_{bb}\phi^0,\rho^0\rangle$
 \medbreak\quad
$+\varepsilon_{\DR,\alpha_1,\alpha_2}^{11}\langle
L_{ab}L_{ab}\phi^0,\rho^0\rangle
+\varepsilon_{\DR,\alpha_1,\alpha_2}^{12}
\langle \phi^0_{;a},\rho^0_{;a}\rangle
+\varepsilon_{\DR,\alpha_1,\alpha_2}^{13}
\langle\tau\phi^0,\rho^0\rangle$
\medbreak$\quad+
\varepsilon_{\DR,\alpha_1,\alpha_2}^{14}
\langle\phi^1,\rho^1\rangle
+\varepsilon_{\RR,\alpha_1,\alpha_2}^{16}
\langle \SR^2\phi^0,\rho^0\rangle
+\varepsilon_{\RR,\alpha_1,\alpha_2}^{17}
\langle \SR\phi^1,\rho^0\rangle$
\medbreak$\left.\vphantom{\vrule height 10pt}\quad
+\varepsilon_{\RR,\alpha_1,\alpha_2}^{18}
\langle\SR\phi^0,\rho^1\rangle
+\varepsilon_{\RR,\alpha_1,\alpha_2}^{19}
\langle \SR L_{aa}\phi^0,\rho^0\rangle \right \}dy$.
\medbreak\noindent We omit the terms involving $\SR$ for Dirichlet boundary conditions.
\end{lemma}

We shall assume the following henceforth; it is properly part of Conjecture~\ref{conj-1.2}
but we postponed it until we could introduce the appropriate notation. 
 Note that it is crucial that we permit $\alpha_i$ to be complex so that
the open set of Equation~(\ref{eqn-1.a}) is connected as this would fail if
we only considered real $\alpha_i$. We refer to \cite{BGS}
where a related result was established when the specific heat is smooth:

\begin{conjecture}\label{conj-1.5}
The functions $\varepsilon_{\DR,\alpha_1,\alpha_2}^\nu$
appearing in Lemma~\ref{lem-1.4} have an analytic extension to the
connected open set defined by Equation~(\ref{eqn-1.a}).
\end{conjecture}

\subsection{The invariant $\beta_0^{\partial M}$}
In Section~\ref{sect-3}, we will make a special case computation on the half-line to
identify the constants $\varepsilon_{\DR,\alpha_1,\alpha_2}$
of Lemma~\ref{lem-1.4}. Set
$$\varepsilon_{\DR}:=\left\{
\begin{array}{lll}
-1&\text{if}&\BB=\BB_\DD\\
+1&\text{if}&\BB=\BB_{\mathcal{R}}
\end{array}\right\}\,.$$

\begin{lemma}\label{lem-1.6}
If $(\alpha_1,\alpha_2)$ satisfies Equation~(\ref{eqn-1.a}), then
\medbreak\quad
$\varepsilon_{\DR,\alpha_1,\alpha_2}
:=\varepsilon_{\DR}\cdot2^{-\alpha_1 -\alpha_2}\pi^{-1/2}
\Gamma\left(\frac{2-\alpha_1-\alpha_2}2\right) \cdot
 \frac{\Gamma(1-\alpha_1)\Gamma(1-\alpha_2)}
 {\Gamma(2-\alpha_1-\alpha_2)}$
\medbreak\qquad
$+ 2^{-\alpha_1-\alpha_2}\pi^{-1/2}\Gamma\left(\frac{2-\alpha_1-\alpha_2}2\right)
 \Gamma(\alpha_1+\alpha_2-1) \cdot\left(\frac{\Gamma(1-\alpha_1)}{
 \Gamma(\alpha_2)}+\frac{\Gamma(1-\alpha_2)}
{ \Gamma(\alpha_1)}\right)$.
\end{lemma}

The following result is an immediate consequence of Lemma~\ref{lem-1.6}
using the usual properties of the $\Gamma$ function. We will give another proof
in Section~\ref{sect-4} that arises from certain functorial properties of the heat
content asymptotics.

\begin{lemma}\label{lem-1.7}
We have the recursion relations:
\begin{enumerate}
\item
$\displaystyle\varepsilon_{\DR,\alpha_1-2,\alpha_2}
=\frac{2(\alpha_1-2)(\alpha_1-1)}{3-\alpha_1-\alpha_2}
\varepsilon_{\DR,\alpha_1,\alpha_2}
$.
\medbreak\item
$\displaystyle\varepsilon_{\DR,\alpha_1,\alpha_2-2}
  =\frac{2(\alpha_2-2)(\alpha_2-1)}{3-\alpha_1-\alpha_2}
 \varepsilon_{\DR,\alpha_1,\alpha_2}
$.
  \medbreak\item
$\displaystyle
\displaystyle\varepsilon_{\DR,\alpha_1-1,\alpha_2-1}
=-\frac{2(\alpha_1-1)(\alpha_2-1)}
{3-\alpha_1-\alpha_2}
\varepsilon_{\RR/\DD,\alpha_1,\alpha_2}
$.
\end{enumerate}
\end{lemma}
Note that the roles of Neumann and Dirichlet boundary conditions are interchanged in Assertion~(3).

\subsection{Heat content asymptotics for Dirichlet boundary conditions}
Dirichlet boundary conditions have been treated previously in
\cite{BG12,BGS}; we summarize those results as follows:

\begin{theorem}
If $(\alpha_1,\alpha_2)$ satisfy Equation~(\ref{eqn-1.a}), then:
\medbreak\quad
$\beta_{0,\alpha_1,\alpha_2}^{\partial M}(\phi,\rho,D,\mathcal{B}_{\mathcal{D}})
=\textstyle\int_{\partial M}
\varepsilon_{\mathcal{D},\alpha_1,\alpha_2}\langle\phi^0,\rho^0\rangle dy$,
\medbreak\quad
$\beta_{1,\alpha_1,\alpha_2}^{\partial M}(\phi,\rho,D,\mathcal{B}_{\mathcal{D}})=\textstyle\int_{\partial M}
\{-{\textstyle\frac12}(\varepsilon_{\mathcal{D},\alpha_1-1,\alpha_2}
+\varepsilon_{\mathcal{D},\alpha_1,\alpha_2-1})L_{aa}\langle\phi^0,\rho^0\rangle$
\smallbreak\qquad
$+\varepsilon_{\mathcal{D},\alpha_1-1,\alpha_2}\langle\phi^1,\rho^0\rangle
+\varepsilon_{\mathcal{D},\alpha_1,\alpha_2-1}\langle\phi^0,\rho^1\rangle\}dy$,
\medbreak\quad
$\beta_{2,\alpha_1,\alpha_2}^{\partial M}(\phi,\rho,D,\mathcal{B}_{\mathcal{D}})=\textstyle\int_{\partial M}
\{-{\textstyle\frac12}(\varepsilon_{\mathcal{D},\alpha_1-2,\alpha_2}
+\varepsilon_{\mathcal{D},\alpha_1-1,\alpha_2-1})L_{aa}\langle\phi^1,\rho^0\rangle$
\smallbreak\qquad
$+ \varepsilon_{\mathcal{D},\alpha_1,\alpha_2}\langle E\phi^0,\rho^0\rangle
+\varepsilon_{\mathcal{D},\alpha_1-2,\alpha_2}\langle\phi^2,\rho^0\rangle
+\varepsilon_{\mathcal{D},\alpha_1,\alpha_2-2}\langle\phi^0,\rho^2\rangle$
\smallbreak\qquad
$-{\textstyle\frac12}(\varepsilon_{\mathcal{D},\alpha_1-1,\alpha_2-1}
+\varepsilon_{\mathcal{D},\alpha_1,\alpha_2-2})L_{aa}\langle\phi^0,\rho^1\rangle$
\smallbreak\qquad
$+(-{\textstyle\frac14}\varepsilon_{\mathcal{D},\alpha_1-2,\alpha_2}
-{\textstyle\frac14}\varepsilon_{\mathcal{D},\alpha_1,\alpha_2-2}
+{\textstyle\frac12}\varepsilon_{\mathcal{D},\alpha_1,\alpha_2})
(L_{ab}L_{ab}+\hbox{\rm Ric}_{mm})\langle\phi^0,\rho^0\rangle$
\smallbreak\qquad
$-\varepsilon_{\mathcal{D},\alpha_1,\alpha_2}\langle\phi_{;a}^0,\rho_{;a}^0\rangle
   +0\tau\phi^0\rho^0
   +\varepsilon_{\mathcal{D},\alpha_1-1,\alpha_2-1}\langle\phi^1,\rho^1\rangle$
   \smallbreak\qquad
   $+(\textstyle\frac18\varepsilon_{\mathcal{D},\alpha_1-2,\alpha_2}
+\frac18\varepsilon_{\mathcal{D},\alpha_1,\alpha_2-2}
+\frac14\varepsilon_{\mathcal{D},\alpha_1-1,\alpha_2-1}
-\frac14\varepsilon_{\mathcal{D},\alpha_1,\alpha_2})L_{aa}L_{bb}
     \langle\phi^0,\rho^0\rangle\}dy$.
\end{theorem}
\subsection{Heat content asymptotics for Robin boundary conditions}
The following is the main result of this paper; it will be established
in Section~\ref{sect-4} using various functorial properties of these invariants:

\begin{theorem}\label{thm-1.9}
If $(\alpha_1,\alpha_2)$ satisfy Equation~(\ref{eqn-1.a}), then:
\medbreak$\beta_{0,\alpha_1,\alpha_2}^{\partial M}
     (\phi,\rho,D,\BB_\RR)=\displaystyle\int_{\partial M}
       \varepsilon_{\RR,\alpha_1,\alpha_2}
\langle\phi^0,\rho^0\rangle dy$.
\medbreak$
\beta_{1,\alpha_1,\alpha_2}^{\partial M}
(\phi,\rho,D,\BB_\RR)
=\displaystyle\int_{\partial M}\bigg\{
\Xone
\langle\phi^1,\rho^0\rangle+\Xthree\langle\phi^0,\rho^1\rangle$
\medbreak\quad$
\Xtwo\langle L_{aa}\phi^0,\rho^0\rangle$
\medbreak\quad$
+\Xfifteen\langle \SR\phi^0,\rho^0\rangle
\bigg\}dy$.
\medbreak$
\beta_{2,\alpha_1,\alpha_2}^{\partial M}
(\phi,\rho,D,\BB_\RR)
    =\displaystyle\int_{\partial M}\bigg\{
    \Xfour
    \langle\phi^2,\rho^0\rangle+\Xseven
\langle\phi^0,\rho^2\rangle$
 \medbreak\quad$\Xfive\langle L_{aa}\phi^1,\rho^0\rangle$
\medbreak\quad$\Xeight\langle L_{aa}\phi^0,\rho^1\rangle$
\medbreak$\quad
    \Xnine
     \langle(\operatorname{Ric}_{mm}+L_{ab}L_{ab})\phi^0,\rho^0\rangle$
\medbreak\quad$
+\{\Xten\}\langle L_{aa}L_{bb}\phi^0,\rho^0\rangle$
 \medbreak\quad$\Xtwelve
\langle\phi_{;a}^0,\rho^0_{;a}\rangle
+\Xthirteen\cdot
\langle\tau\phi^0,\rho^0\rangle$
\medbreak$\quad+
\Xsix\langle E
    \phi^0,\rho^0\rangle+\Xfourteen
\langle\phi^1,\rho^1\rangle$
\medbreak$\quad+\{\Xsixteen\}
\langle \SR^2\phi^0,\rho^0\rangle$
\medbreak\quad$+\Xseventeen\langle\SR\phi^1,\rho^0\rangle$
\medbreak\quad
$+\Xeighteen\langle\SR\phi^0,\rho^1\rangle$
\medbreak\quad$
+\left\{\Xnineteen\right\}\langle\SR L_{aa}\phi^0,\rho^0\rangle\bigg\}dy$.
\end{theorem}

\subsection{Outline of the paper}
Our purpose is to a large extent expository so
we shall give complete details rather than simply referring to previous
papers. In Section~\ref{sect-2}, we discuss the local
invariants we shall need. In Section~\ref{sect-3}, we make a special case computation on the
half-line to determine $\beta_0^{\partial M}$. In Section~\ref{sect-4}, we use functorial properties
of these invariants to determine $\beta_1^{\partial M}$ and $\beta_2^{\partial M}$.
We must work in great generality in the context of operators of Laplace type
in order to apply the functorial methods of Section~\ref{sect-4}. It is one of
the paradoxes in this subject that even if one is only interested in the scalar
Laplacian in flat space, it is necessary to derive completely general formulas
and then specialize them.

\section{Local invariants}\label{sect-2}
By assumption, we may express the asymptotic coefficients $\beta_{j,\alpha_1,\alpha_2}^{\partial M}$ as the integrals
of local invariants over the boundary:
\begin{equation}\label{eqn-2.a}
\beta_{j,\alpha_1,\alpha_2}^{\partial M}(\phi,\rho,D,\BB_\DR)=
   \int_{\partial M}\beta_{j,\alpha_1,\alpha_2}^{\partial M}(\phi,\rho,D,\BB_\DR)(y)dy\,.
\end{equation}
Fix a system of local coordinates $x=(x^1,...,x^m)$. Set $\partial_x^{\vec a}=(\partial_{x_1})^{a_1}...(\partial_{x_m})^{a_m}$. We have
\begin{equation}\label{eqn-2.b}
\beta_{j,\alpha_1,\alpha_2}^{\partial M}(\phi,\rho,D,\BB_\DR)(y)
=\beta_{j,\alpha_1,\alpha_2}^{\partial M}(\phi^\ell,\rho^\ell,\partial_x^{\vec a}g_{ij},\partial_x^{\vec b}A^k,\partial_x^{\vec c}B,
\partial_x^{\vec d}\SR)\,.
\end{equation}
Here we need only consider a finite number of jets; we omit the $\SR$ variables for Dirichlet boundary conditions and for
Robin boundary conditions only differentiate $\SR$ tangentially.
The local invariants are polynomial in the derivatives of the structures
involved with coefficients that depend
smoothly on the metric.

\subsection{A coordinate formulation}
We summarize the arguments briefly - this also serves
to motivate the power $(1+j-\alpha_1-\alpha_2)/2$ of $t$ in
Conjecture~\ref{conj-1.2}. Fix a point $P\in\partial M$ and choose local coordinates
$\vec y=(y^1,...,y^{m-1})$ on the boundary which are centered
at $P$. Let $\vec x:=(\vec y,r)$ be the associated adapted coordinates.
Let indices $a,b,...$ range from $1$ to $m-1$.
We then have
$$ds^2=g_{ab}dy^a\circ dy^b+dr\circ dr\text{ where }g_{ab}=g_{ab}(y,r)\,.$$
We further
normalize the choice of coordinates so that $g_{ab}(P)=\delta_{ab}$.
This eliminates
the smooth dependence on the metric tensor and ensures that the invariants of
Equation~(\ref{eqn-2.b}) are polynomial in the jets of the structures.
For a smooth function $f$ and a multi-index $\vec a=(a_1,...,a_m)$ we
let
$$f_{\vec a}=\partial_{x_1}^{a_1}...\partial_{x_m}^{a_m}f\,.$$
We define
$$\begin{array}{lll}
\operatorname{ord}(g_{ij,\vec a}):=|\vec a|,&
\operatorname{ord}(A^k_{\vec b}):=|\vec b|+1,&
\operatorname{ord}(B_{\vec c}):=|\vec c|+2,\\
\operatorname{ord}(\phi^\ell):=\ell,&
\operatorname{ord}(\rho^\ell):=\ell,&
\operatorname{ord}(S_{\mathcal{R},\vec d}):=|\vec d|+1\,.\vphantom{\vrule height 11pt}
\end{array}$$
Here, since $\SR$ is only defined on the boundary, $\vec d$ only reflects tangential indices;
we delete the $S_{\mathcal{R},\vec d}$ variables from consideration
for Dirichlet boundary conditions as they play no role.

\begin{lemma}\label{lem-2.1}
$\beta_{j,\alpha_1,\alpha_2}^{\partial M}(\phi,\rho,D,\BB)(y)$
is homogeneous of weighted order $j$.
\end{lemma}

\begin{proof}
Fix $P\in\partial M$ and choose adapted coordinates so $g_{ab}(P)=\delta_{ab}$. Then  the local formula for
$\beta_{j,\alpha_1,\alpha_2}^{\partial M}$
is polynomial in the variables
$$\{g_{ij/\vec a}, A^k_{\vec b},B_{\vec c},\phi^\ell,\rho^\ell,S_{\mathcal{R},\vec d}\}_{|\vec a|>0}\,.$$
We use dimensional analysis. Let $c>0$,
and consider the operator $D_c:=c^{-2}D$.
It is then clear in a purely formal sense that
$e^{-t(c^{-2}D)}=e^{-(c^{-2}t)D}$ and thus if
$T_c$ is the temperature distribution defined by $D_c$, then
\begin{equation}\label{eqn-2.c}
T_c(x;t)=T(x;c^{-2}t)\,.
\end{equation} We justify this formal computation by
verifying that the relations of
Equation~(\ref{eqn-1.c}) are satisfied. Suppose that $T_c$ is defined by Equation~(\ref{eqn-2.c}). Then:
$$
(\partial_t+D_c)T_c(x;t)=c^{-2}\left\{\partial_tT+DT\right\}(x;c^{-2}t)=0\ \text{ and }\
\lim_{t\downarrow0}T(\cdot;c^{-2}t)=\phi(\cdot)\,.
$$
We must check the boundary conditions are satisfied by $T_c$; this is immediate with Dirichlet
boundary conditions so we set $\BB_{\mathcal{D}_c}:=\BB_{\mathcal{D}}$ and ignore the subscript.
The situation concerning Robin boundary conditions requires more work.
We use Lemma~\ref{lem-1.1} to see that $D_c$ and $D$
determine the same connection. However the normal rescales; $e_m(c^2g)=c^{-1}e_m$. Set:\begin{equation}\label{eqn-2.d}
\mathcal{B}_{\mathcal{R}_c}:=\nabla_{c^{-1}e_m}+S_{\mathcal{R},c}\text{ where }
S_{\mathcal{R},c}:=c^{-1}\SR\,.
\end{equation}
We complete the proof that $T_c=e^{-tD_c}\phi$ by verifying $T_c$ satisfies the rescaled
Robin boundary conditions:
$$
\mathcal{B}_{\mathcal{R}_c}T_c=
\left.c^{-1}\{(\nabla_{e_m}+\SR)T(\cdot,c^{-2}t)\}\right|_{\partial M}
=c^{-1}\mathcal{B}_{\mathcal{R}}T(\cdot,c^{-2t})=0\,.
$$

Since $g_c=c^2g$, $dx_c=c^mdx$. We verify that the interior invariants $\beta_n^M$ rescale
properly:
\begin{eqnarray}\label{eqn-2.e}
&&\beta_n^M(\phi,\rho,D_c)=(-1)^n/n!\cdot\mathcal{I}_{Reg}^{g_c}\{\langle D_c^n\phi,\rho\rangle\}\\
&=&(-1)^n/n~\cdot c^m\mathcal{I}_{Reg}^{g}\{\langle (c^{-2}D)^n\phi,\rho\rangle\}
=(-1)^n/n~\cdot c^{m-2n}\mathcal{I}_{Reg}^{g}\{\langle D^n\phi,\rho\rangle\}\nonumber\\
&=&c^{m-2n}\beta_n^M(\phi,\rho,D)\,.\nonumber
\end{eqnarray}
We apply Conjecture~\ref{conj-1.2} and use Equation~(\ref{eqn-2.e}) to expand:
\begin{eqnarray}
\beta(\phi,\rho,D_c,{\BB_\DR}_c)(t)
&\sim&\displaystyle\sum_{n=0}^\infty t^n
        c^{m-2n}\beta_n^M(\phi,\rho,D)\nonumber\\
&+&\displaystyle\sum_{j=0}^\infty t^{(1+j-\alpha_1-\alpha_2)/2}
   \beta_{j,\alpha_1,\alpha_2}^{\partial
   M}(\phi,\rho,D_c,{\BB_\DR}_c)\,.\label{eqn-2.f}
\end{eqnarray}
We apply Equation~(\ref{eqn-2.c}) to see:
\begin{eqnarray}
&&\beta(\phi,\rho,D_c,{\BB_\DR}_c)(t)=\displaystyle\int_M\langle T_c(x;t),\rho(x)\rangle dx_c
=c^m\int_M\langle T(x;c^{-2}t),\rho(x)\rangle dx\nonumber\\
&&\qquad=c^m\beta(\phi,\rho,D,\BB_\DR)(c^{-2}t)
\sim c^m\displaystyle\sum_{n=0}^\infty(c^{-2}t)^n\beta_n^M(\phi,\rho,D)\nonumber\\
    &&\qquad\qquad\qquad\qquad+c^m\sum_j(c^{-2}t)^{(1+j-\alpha_1-\alpha_2)/2}
   \beta_{j,\alpha_1,\alpha_2}^{\partial
   M}(\phi,\rho,D,\BB_\DR)\,.\label{eqn-2.g}
\end{eqnarray}
\medbreak\noindent We equate powers of $t$ in the asymptotic expansions of Equation~(\ref{eqn-2.f})
and Equation~(\ref{eqn-2.g}) to see:
\begin{equation}\label{eqn-2.h}
\beta_{j,\alpha_1,\alpha_2}^{\partial M}
(\phi,\rho,D_c,{\BB_\DR}_c)=c^{m+\alpha_1+\alpha_2-1-j}
\beta_{j,\alpha_1,\alpha_2}^{\partial M}
(\phi,\rho,D,\BB_\DR)\,.
\end{equation}
We use Equation~(\ref{eqn-2.a}). Since $dy_c=c^{m-1}dy$,
Equation~(\ref{eqn-2.h}) implies:
\medbreak
$\beta_{j,\alpha_1,\alpha_2}^{\partial M}(\phi,\rho,D_c,\BB_{{\DR}_c})=c^{m-1}\displaystyle\int_{\partial M}\beta_{j,\alpha_1,\alpha_2}^{\partial M}
(\phi,\rho,D_c,{\BB_\DR}_c)(y)dy$
\medbreak\ \
$=
c^{m-1-j+\alpha_1+\alpha_2}\beta_{j,\alpha_1,\alpha_2}^{\partial M}
(\phi,\rho,D,\BB_\DR)$
\medbreak\ \
$\displaystyle=c^{m-1-j+\alpha_1+\alpha_2}\int_{\partial M}\beta_{j,\alpha_1,\alpha_2}^{\partial M}
(\phi,\rho,D,\BB_\DR)(y)dy$.
\medbreak\noindent
Consequently (ignoring divergence terms), we have the following relation for the local invariants:
\begin{equation}\label{eqn-2.i}
\beta_{j,\alpha_1,\alpha_2}^{\partial M}(\phi,\rho,D_c,{\BB_\DR}_c)(y)
=c^{-j+\alpha_1+\alpha_2}\beta_{j,\alpha_1,\alpha_2}^{\partial M}
(\phi,\rho,D,\BB_\DR)(y)\,.\end{equation}

We wish to apply the formalism of Equation~(\ref{eqn-2.b}). We may expand:
$$
\phi\sim r_c^{-\alpha_1}(\phi_c^0+r_c\phi_c^1+...)\sim r^{-\alpha_1}(\phi^0+r\phi^1+...)\,.
$$
Equating terms in the asymptotic series, using the fact that $r_c=c\cdot r$, and arguing similarly with $\rho$ yields:
\begin{equation}\label{eqn-2.j}
\phi_c^\ell=c^{\alpha_1-\ell}\phi^\ell\quad\text{and}\quad
\rho_c^\ell=c^{\alpha_2-\ell}\rho^\ell\,.
\end{equation}
We renormalize the coordinate system and set
$\vec x_c=c\cdot\vec x$; $\partial_{x_{c,i}}=c^{-1}\partial_{x_i}$.
We introduce the coordinate systems $\vec x$ and $\vec x_c$ into the notation as
a computational aid; the local invariants $\beta_{j,\alpha_1,\alpha_2}^{\partial M}(y)$
do not, of course, depend on the choice of the coordinate system. We compute:
$$
g_{c,ij}(\vec x_c,P)=c^2g_{ij}(c^{-1}\partial_{x_i},c^{-1}\partial_{x_j})(\vec x,P)
=c^{-2}c^2g_{ij}(\vec x,P)=\delta_{ij}\,.
$$
Consequently, $g_{c,ij}(\vec x_c)=g_{ij}(\vec x)$ and we have that:
$$
g_{c,ij,\vec a}(\vec x_c,P)=c^{-|\vec a|}g_{ij,\vec a}(\vec x,P)\,.
$$
We expand:
\begin{eqnarray*}
D_c&=&c^{-2}(g^{ij}\partial_{x_i}\partial_{x_j}+A^k\partial_{x_k}+B)\\
&=&-(g^{ij}\partial_{x_{c,i}}\partial_{x_{c,j}}+c^{-1}A^k\partial_{x_{c,k}}+c^{-2}B)\\
&=&-(g^{ij}\partial_{x_{c,i}}\partial_{x_{c,j}}+A_c^k\partial_{x_{c,k}}+B_c)
\end{eqnarray*}
to see that
$A_c^k(\vec x_c,P)=c^{-1}A^k(\vec x,P)$ and $B_c(\vec x_c,P)=c^{-2}B(\vec x,P)$.
 Consequently
\begin{equation}\label{eqn-2.k}
A_{c,\vec b}^k(\vec x_c,P)=c^{-1-|\vec b|}A_{\vec b}^k(\vec x,P)\quad\text{and}\quad
B_{c,\vec c}(\vec x_c,P)=c^{-2-|\vec c|}B_{\vec c}(\vec x,P)\,.
\end{equation}
Similarly, by Equation~(\ref{eqn-2.d}) we have:
\begin{equation}\label{eqn-2.L}
S_{\mathcal{R},c,\vec d}(\vec x_c,P)=c^{-1-|\vec d|}S_{\mathcal{R},\vec d}(\vec x,P)\,.
\end{equation}

We use Equation~(\ref{eqn-2.j}), Equation~(\ref{eqn-2.k}), and
Equation~(\ref{eqn-2.L}) together with the observation that $\beta_j^{\partial M}$ is
bilinear in $(\phi,\rho)$ to see:
\begin{eqnarray}
&&\beta_{j,\alpha_1,\alpha_2}^{\partial M}(\phi,\rho,D_c,\BB_\DR)(\vec x_c,P)
\label{eqn-2.m}\\
&=&\beta_{j,\alpha_1,\alpha_2}^{\partial M}
(c^{\alpha_1-\ell}\phi^\ell,c^{\alpha_2-\ell}\rho^\ell,
c^{-|\vec a|}g_{ij,\vec a},c^{-1-|\vec b|}A^i_{\vec b},
c^{-2-|\vec c|}B_{\vec c},\\
&&\qquad\qquad\qquad c^{-1-|\vec d|}S_{\mathcal{R},\vec d})(\vec x,P)\nonumber\\
&=&c^{\alpha_1+\alpha_2}\beta_{j,\alpha_1,\alpha_2}^{\partial M}
(c^{-\ell}\phi^\ell,c^{-\ell}\rho^\ell,
c^{-|\vec a|}g_{ij,\vec a},c^{-1-|\vec b|}A^i_{\vec b},
c^{-2-|\vec c|}B_{\vec c},\\
&&\qquad\qquad\qquad c^{-1-|\vec d|}S_{\mathcal{R},\vec d})(\vec x,P)\,.\nonumber
\end{eqnarray}
We use Equation~(\ref{eqn-2.i}) and Equation~(\ref{eqn-2.m}) to conclude therefore:
\begin{eqnarray*}
&&\beta_{j,\alpha_1,\alpha_2}^{\partial M}
(c^{-\ell}\phi^\ell,c^{-\ell}\rho^\ell,
c^{-|\vec a|}g_{ij/\vec a},c^{-1-|\vec b|} A^k_{\vec b},
c^{-2-|\vec c|}B_{\vec c},c^{-1-|\vec d|}S_{\mathcal{R}/\vec d})(\vec x,P)\\
&&=c^{-j}
\beta_{j,\alpha_1,\alpha_2}^{\partial M}
(\phi^\ell,\rho^\ell,
g_{ij/\vec a}, A^k_{\vec b},
B_{\vec c},S_{\mathcal{R},\vec d})(\vec x,P)\,.
\end{eqnarray*}
It now follows that $\beta_{j,\alpha_1,\alpha_2}^{\partial M}$ is homogeneous of weighted
order $j$. 
 There can be divergence terms which integrate to zero on the boundary and which
are not controlled by this analysis. These terms play no role in the asymptotic coefficients and
may therefore be ignored.
\end{proof}

\subsection{A tensorial formulation}\label{sect-2.2}
We have just discussed the
homogeneity property of the invariants $\beta_j$ from a coordinate point of view.
We now consider the same property from a more invariant point of view.
Fix a point $P\in\partial M$. Choose geodesic coordinates on the boundary centered at $P$.
Form an adapted coordinate system for $P$ in $M$.
Then the only non-zero first derivatives of the metric are
$g_{ab,m}(P)=\partial_{x_m}g_{ab}(P)$. The {\it second fundamental form} is given by setting:
$$L_{ab}:=g(\nabla_{\partial_y^a}\partial_y^b,\partial_r)=\Gamma_{abm}
=-\textstyle\frac12g_{ab,m}\,.$$
Keeping in mind the necessity to preserve the condition $g_{ab}(P)=\delta_{ab}$
and observing that the Levi-Civita connection is unchanged by rescaling,
we see that:
$$L_{ab}(g_c)=c^2g(\nabla_{c^{-1}\partial_{y_a}}c^{-1}\partial_{y_b},c^{-1}\partial_r)
=c^{-1}L_{ab}(g)\,.$$
Thus $L$ is homogeneous of degree $1$ under rescaling. Similarly we have:
\begin{eqnarray*}
R_{ijkl}(g_c)(P)&=&c^2g((\nabla_{c^{-1}\partial_{x_i}}\nabla_{c^{-1}\partial_{x_j}}
-\nabla_{c^{-1}\partial_{x_j}}\nabla_{c^{-1}\partial_{x_i}})c^{-1}\partial_{x_k},c^{-1}
\partial_{x_l})\\
&=&c^{-2}R_{ijkl}(g)(P),
\end{eqnarray*}
so that $R_{ijkl}$ has order 2.  Arbitrary partial derivatives of the metric at $P$ may
now be expressed in terms of these tensors and their covariant derivatives;
there are universal curvature
relations (see, for example, the discussion in \cite{GPS12}) but these play no role.
Rather than considering tangential derivatives of $\SR$, we consider
covariant derivatives of $\SR$ with respect to the connection defined by $D$
and the Levi-Civita
connection on the boundary. Let $\Omega_{ij}$ be the curvature of the
connection defined by $D$. Let `;' denote the components of multiple covariant differentiation
with respect to the Levi-Civita connection of $M$ and the connection $\nabla$ defined by $D$
and let `:' defined similarly using the Levi-Civita connection of $\partial M$ and $\nabla$; we
can only differentiate $L$ and $\SR$ tangentially. Each covariant derivative adds $1$ to
the order. The derivatives of the total symbol of $D$ may
then be expressed in terms of these variables so we may regard
$\beta_{j,\alpha_1,\alpha_2}(\phi,\rho,D,\mathcal{B})(y)$
as a polynomial
in the variables
$$
\{\phi^\ell,\rho^\ell,L_{a_1a_2:a_3...a_k},R_{i_1i_2i_3i_4;i_5...i_k},
E_{;i_1...i_k},S_{\mathcal{R}:b_1...b_k},\Omega_{i_1i_2;i_3...i_k}\}
$$
which is invariantly defined and which is homogeneous of total order $j$.
H. Weyl's theorem \cite{W46} then lets us write this in terms of
contractions of indices; the structure group is the orthogonal group $O(m-1)$ so the index $m$
plays a distinguished role.
\subsection{The proof of Lemma~\ref{lem-1.4}}
Lemma~\ref{lem-1.4} now follows from Lemma~\ref{lem-2.1} and
the discussion in Section~\ref{sect-2.2} once a suitable basis of Weyl invariants is written down.
We integrate by parts and ignore divergence terms to replace the invariants
$\int_{\partial M}\langle\phi^0_{:aa},\rho^0\rangle dy$ and
$\int_{\partial M}\langle\phi^0,\rho_{:aa}^0\rangle dy$ by
$-\int_{\partial M}\langle\phi^0_{:a},\rho^0_{:a}\rangle dy$.\hfill\qed

\section{The proof of Lemma~\ref{lem-1.6}}\label{sect-3}
We must determine
the coefficient $\varepsilon_{\DR,\alpha_1,\alpha_2}$ of Lemma~\ref{lem-1.4}
which describes $\beta_{0,\alpha_1,\alpha_2}^{\partial M}(\phi,\rho,D,\BB)$.
We begin with a computation on the strip:
$$
\mathcal{S}:=\{(\alpha_1,\alpha_2)\in\mathbb{R}^2: \alpha_1<1,\
\alpha_2<1,\ \alpha_1+\alpha_2>1\}\,.
$$
The Dirichlet setting was discussed in \cite{BG12,BGG12} so we  concentrate
on Robin boundary conditions; we may set $\SR=0$ as this plays no role in
$\beta_0^{\partial M}$ and simplify
matters by considering Neumann boundary conditions. Let $M=[0,\infty)$, let $D=-\partial_x^2$,
and let $K_\DN(x,y;t)$ be the Dirichlet or Neumann heat kernel on $M$ for $t>0$. Let $\chi$ be a plateau
function which is identically equal to $1$ near $x=0$ and which has
compact support in $[0,\epsilon)$ for $\epsilon$ small. Set
\begin{equation}\label{eqn-3.a}
\phi_{\alpha_1}(x):=x^{-\alpha_1}\chi(x)
\quad\text{and}\quad
\rho_{\alpha_2}(x):=x^{-\alpha_2}\chi(x)\,.
\end{equation}
We shall use this notation subsequently as well.
We have:
$$\phi^\ell=\left\{\begin{array}{ll}1&\text{ if }\ell=0\\0&\text{ if }\ell>0\end{array}\right\}\text{ and }
    \rho^\ell=\left\{\begin{array}{ll}1&\text{ if }\ell=0\\0&\text{ if }\ell>0\end{array}\right\}\,.
    $$
The heat content on $[0,\infty)$ is:
$$
\beta(\phi_{\alpha_1},\rho_{\alpha_2},D,\BB_\DN)(t) =\int_0^\infty\int_0^\infty
K_{\BB_\DN}(x_1,x_2;t)\phi_{\alpha_1}(x_1)\rho_{\alpha_2}(x_2)dx_1dx_2.
$$
Then $\beta(\phi_{\alpha_1},\rho_{\alpha_2},D,\BB_{\mathcal{N}})(t)$ is finite for all $t>0$
if and only if $\alpha_1<1$ and $\alpha_2<1$. We also suppose that
$\alpha_1+\alpha_2>1$. In \cite{BG12} we have shown that
$$
\beta(\phi_{\alpha_1},\rho_{\alpha_2},D,\BB_\DD)(t)=
\varepsilon_{\DD,\alpha_1,\alpha_2}t^{(1-\alpha_1-\alpha_2)/2}
+O(1)\,.
$$
In the special case of the half-line we have that
$$K_{\mathcal B_N}(x_1,x_2;t)=K_{\mathcal B_D}(x_1,x_2;t)+
2(4\pi)^{-1/2}e^{-(x_1+x_2)^2/(4t)}\,.
$$
We compute:
\begin{eqnarray}
&&\beta(\phi_{\alpha_1},\rho_{\alpha_2},D,\BB_{\mathcal{N}})(t)
=\beta(\phi_{\alpha_1},\rho_{\alpha_2},D,\BB_\DD)(t)\nonumber \\
&&\quad+2(4\pi
t)^{-1/2}\int_0^\infty\int_0^\infty
e^{-(x_1+x_2)^2/(4t)}x_1^{-\alpha_1} x_2^{-\alpha_2}dx_1dx_2\label{eqn-3.b}\\
&&\quad +2(4\pi t)^{-1/2}\int_0^\infty\int_0^\infty
e^{-(x_1+x_2)^2/(4t)}x_1^{-\alpha_1}
x_2^{-\alpha_2}(\chi(x_1)\chi(x_2)-1)dx_1dx_2.
\nonumber\end{eqnarray}
The second term in the right hand side above is evaluated by a
change of variable $x_2=x_1\sigma$. By Tonelli's Theorem we have that it
equals \medbreak\quad $2(4\pi
t)^{-1/2}\displaystyle\int_0^\infty\int_0^\infty
e^{-(x_1+x_2)^2/4t}
       x_1^{-\alpha_1} x_2^{-\alpha_2}dx_1dx_2$
\medbreak\quad $=2\displaystyle(4\pi
t)^{-1/2}\int_0^\infty\int_0^\infty x_1^{1-\alpha_1-\alpha_2}
\sigma^{-\alpha_2}e^{-x_1^2(1+\sigma^2)/4t}d\sigma dx_1$
\medbreak\quad $=2(4\pi
t)^{-1/2}\displaystyle\int_0^\infty\int_0^\infty
\sigma^{-\alpha_2}x_1^{1-\alpha_1-\alpha_2}
e^{-x_1^2(1+\sigma)^2/4t}dx_1d\sigma$ \medbreak\quad
$=2^{1-\alpha_1-\alpha_2}\pi^{-1/2}\displaystyle
\Gamma\left(\frac{2-\alpha_1-\alpha_2}2\right)
t^{(1-\alpha_1-\alpha_2)/2}
\int_0^\infty\sigma^{-\alpha_2}(1+\sigma)^{\alpha_1+\alpha_2-2}d\sigma.$
\medbreak\noindent Note that if $\alpha_1<1$ and $\alpha_2<1$ then
$$
\int_0^\infty\sigma^{-\alpha_2}(1+\sigma)^{\alpha_1+\alpha_2-2}d\sigma
=\int_0^1(\sigma^{-\alpha_1}+\sigma^{-\alpha_2})
(1+\sigma)^{\alpha_1+\alpha_2-2}d\sigma\,.
$$
We
will show that the third term in the right hand side of Equation~(\ref{eqn-3.b})
 vanishes up to all orders. First note that there exists
$\epsilon>0$ so $\chi(x)=1, 0\le x \le \epsilon$. Secondly
note that there exists a constant $C<\infty$ so $|\chi(x)|\le C, x\ge
0$. Hence
$$|\chi(x_1)\chi(x_2)-1|\le
(C^2+1)1_{[\epsilon,\infty)}(x_1)1_{[\epsilon,\infty)}(x_2)\,.
$$
 It
follows that the third term in the right hand side of Equation (\ref{eqn-3.b})
is bounded in absolute value by
\medbreak\quad
$\displaystyle2(C^2+1)(4\pi t)^{-1/2}\int_{\epsilon}^\infty\int_{\epsilon}^\infty
e^{-(x_1+x_2)^2/(4t)}x_1^{-\alpha_1} x_2^{-\alpha_2}dx_1dx_2$

\medbreak\quad
$\displaystyle\quad\le(C^2+1)t^{-1/2}e^{-\epsilon^2/(4t)}\int_0^\infty\int_0^\infty
e^{-(x_1^2+x_2^2)/(8t)}x_1^{-\alpha_1}x_2^{-\alpha_2}dx_1dx_2$
\medbreak\quad
$\displaystyle\quad=O(e^{-\epsilon^2/(5t)})$.

\medbreak By Lemma~\ref{lem-1.4}, we have
$\beta_{0,\alpha_1,\alpha_2}^{\partial M}(\phi,\rho,D,\BB)=\displaystyle\int_{\partial M}
\varepsilon_{\DR,\alpha_1,\alpha_2}\langle\phi^0,\rho^0\rangle dy$.
The above discussion permits us to decompose:
\medbreak\quad
$\displaystyle\varepsilon_{\mathcal{B}_\DR,\alpha_1,\alpha_2}
=\varepsilon_{\mathcal{B}_\DR}\kappa_{\alpha_1,\alpha_2}+\tilde\kappa_{\alpha_1,\alpha_2}$ where
\medbreak\quad
$\displaystyle\kappa_{\alpha_1,\alpha_2}
:=2^{-\alpha_1 -\alpha_2}\pi^{-1/2}
\Gamma(\frac{2-\alpha_1-\alpha_2}2)
\int_0^1
(\sigma^{-\alpha_1}+\sigma^{-\alpha_2})(1+\sigma)^{\alpha_1+\alpha_2-2}d\sigma$
\medbreak\quad
$\displaystyle\tilde\kappa_{\alpha_1,\alpha_2}
:=2^{-\alpha_1 -\alpha_2}\pi^{-1/2}
\Gamma(\frac{2-\alpha_1-\alpha_2}2)
\int_0^1(\sigma^{-\alpha_1}+\sigma^{-\alpha_2})(1-\sigma)^{\alpha_1+\alpha_2-2}d\sigma$.
\medbreak\noindent
The integral defining $\kappa_{\alpha_1,\alpha_2}$ clearly is well defined
for $\alpha_1<1$ and $\alpha_2<1$. By replacing $\sigma$ by $\frac1\sigma$, we see that the integral with
respect to $\sigma$ over $[0,1]$ defining
$\kappa_{\alpha_1,\alpha_2}$ is equal to the integral with respect to
$\sigma$ over $[1,\infty)$. Thus the integral with respect to $\sigma$ over
$[0,1]$ is $\frac12$ the integral over $[0,\infty)$. We may now use
 Gradshteyn and Ryzhik \cite{grad65} (see 3.251.11) and
 Gradshteyn and  Ryzhik \cite{grad65} (see 8.3801 and 8.384.1) to evaluate
 these integrals and establish Lemma~\ref{lem-1.6} on the strip.
 By Conjecture~\ref{conj-1.5}, we may then use analytic continuation
 to establish Lemma~\ref{lem-1.6} for the full parameter range given in Equation~(\ref{eqn-1.a}).\hfill\qed

\section{The proof of Theorem~\ref{thm-1.9}: functorial properties}\label{sect-4}
In Section~\ref{sect-4}, we determine the unknown coefficients
$\varepsilon_{\BB,\alpha_1,\alpha_2}^\nu$
of Lemma~\ref{lem-1.4} to complete the proof of Theorem~\ref{thm-1.9}.
\subsection{Direct sum and product formulas}
\begin{lemma}
The coefficients of Lemma~\ref{lem-1.4} are independent
of the fiber dimension of the underlying vector bundle.
\end{lemma}

\begin{proof}
Let $V=V_1\oplus V_2$ be a direct sum vector bundle over $M$, let $D=D_1\oplus D_2$, let
$\phi=\phi_1\oplus\phi_2$, and let  $\rho=\rho_1\oplus\rho_2$. With
Robin boundary conditions, we take $\SR=S_{\mathcal{R},1}\oplus S_{\mathcal{R},2}$. The problem
decouples and consequently
$$\beta(\phi,\rho,D,\BB)(t)=\beta(\phi_1,\rho_1,D_1,\BB)(t)
+\beta(\phi_2,\rho_2,D_2,\BB)(t)\,.$$
 This implies that:
$$\beta_{j,\alpha_1,\alpha_2}^{\partial M}(\phi,\rho,D,\BB)
=\beta_{j,\alpha_1,\alpha_2}^{\partial M}(\phi_1,\rho_1,D_1,\BB)
+\beta_{j,\alpha_1,\alpha_2}^{\partial M}(\phi_2,\rho_2,D_2,\BB)\,.
$$
The desired result now follows; again, we ignore divergence terms when passing to the local formulas.
\end{proof}

Suppose that $(M_1,g_1)$ is a closed Riemannian manifold of dimension $m_1$
and that $(M_2,g_2)$ is a Riemannian manifold with boundary of dimension $m_2$.
We form $(M,g):=(M_1\times M_2,g_1\oplus g_2)$. If $D_i$ are operators of
Laplace type on $M_i$, we form $D=D_1\otimes 1+1\otimes D_2$ on
$V:=V_1\otimes V_2$.
\begin{lemma}\label{lem-4.2}
Adopt the notation established above. We have:
\begin{enumerate}
\item $\beta(\phi,\rho,D,\BB_\DR)(t)=\beta(\phi_1,\rho_1,D_1)(t)\cdot
\beta(\phi_2,\rho_2,D_2,\BB_\DR)(t)$.
\medbreak\item
$\displaystyle\beta_{j,\alpha_1,\alpha_2}^{\partial M}(\phi,\rho,D,\BB_\DR)(x_1,y_2)$
\medbreak$\displaystyle
=\sum_{k+2\ell=j}(-1)^\ell\textstyle\frac1{\ell!}
\langle D^\ell\phi_1,\rho_1\rangle(x_1)
\cdot
\beta_{k,\alpha_1,\alpha_2}^{\partial M}(\phi_2,\rho_2,D_2,\BB_\DR)(y_2)$.
\item The coefficients of Lemma~\ref{lem-1.4} are independent of the dimension $m$.
\item
$\varepsilon_{\DR,\alpha_1,\alpha_2}^6
=\varepsilon_{\DR,\alpha_1,\alpha_2}$,
$\varepsilon_{\DR,\alpha_1,\alpha_2}^{13}=0$,
$\varepsilon_{\DR,\alpha_1,\alpha_2}^{12}
=-\varepsilon_{\DR,\alpha_1,\alpha_2}$.
\end{enumerate}\end{lemma}

\begin{proof} One has purely formally that
\begin{equation}\label{eqn-4.a}
e^{-tD_{\BB_\DR}}\phi=e^{-tD_1}\phi_1\otimes e^{-tD_{2,\BB_\DR}}\phi_2\,.
\end{equation}
We verify Equation~(\ref{eqn-4.a}) using Equation~(\ref{eqn-1.c}) as follows. Let $T_i$
be the temperature distributions defined by $\phi_i$ on $M_i$. We compute:
\begin{eqnarray*}
&&(\partial_t+D)(T_1\otimes T_2)=\{(\partial_t+D_1)T_1\}\otimes T_2
+T_1\otimes\{(\partial_t+D_2)T_2\}=0,\\
&&\lim_{t\downarrow0}\{T_1\otimes T_2)(\cdot;t)=
\lim_{t\downarrow0}\{T_1(\cdot;t)\}\otimes\lim_{t\downarrow0}\{T_2(\cdot;t)\}
=\phi_1\otimes\phi_2,\\
&&\BB_\DR(\phi_1\otimes\phi_2)=\phi_1\otimes\{\BB_\DR\phi_2\}=0\,.
\end{eqnarray*}
Assertion~(1) is now immediate. Since the boundary conditions play no role on
the closed manifold $M_1$, we have
$$\beta(\phi_1,\rho_1,D_1)(t)\sim\sum_{n=0}^\infty t^n\frac{(-1)^n}{n!}
\int_{M_1}\langle D^n\phi_1,\rho_1\rangle dx_1\,.$$
Assertion~(2) now follows from Assertion~(1) by equating terms in the asymptotic
expansions.

We argue as follows to prove Assertion~(3). Let $M_1=S^1$, let
$D_1=-\partial_\theta^2$, and let
 $\phi_1=\rho_1=1$. Since $T_1(x_1;t)=1$,
$$\beta(\phi_1,\rho_1,D_1)(t)=2\pi\,.$$
We use Assertion~(2) to see that:
$$\beta_{j,\alpha_1,\alpha_2}^{\partial M}(\phi,\rho,D,\BB_\DR)
=2\pi\cdot
\beta_{j,\alpha_1,\alpha_2}^{\partial M_2}(\phi_2,\rho_2,D_2,\BB_\DR)\,.$$
This implies that for the local formulas it does not matter if we compute on $M$
or on $M_2$. Consequently the coefficients of Lemma~\ref{lem-4.2}
are independent of $m$.

We complete the proof by considering Assertion~(4). With Robin boundary conditions, we take $\SR=0$
as it plays no role in the evaluation of the coefficients we are considering. Let $(M_1,g_1,\phi_1,\rho_1,D_1)$ be arbitrary.
Let $M_2=[0,\pi]$ with the usual flat metric, let $D_2=\partial_r^2$, and let
$\phi_{\alpha_1}$ and $\rho_{\alpha_2}$ be as given in Equation~(\ref{eqn-3.a}).
Since the structures on $M_2$ are flat,
$$\beta_{j,\alpha_1,\alpha_2}^{\partial M_2}(\phi_2,\rho_2,D_2,\BB_\DR)=\left\{
\begin{array}{ll}
\varepsilon_{\DR,\alpha_1,\alpha_2}&\text{ if }j=0\\
0&\text{ if }j>0
\end{array}\right\}\,.$$
Consequently Assertion~(2) yields:
\begin{eqnarray}
&&\beta_{2,\alpha_1,\alpha_2}^{\partial M}(\phi,\rho,D,\BB_\DR)
=-\varepsilon_{\DR,\alpha_1,\alpha_2}
\int_{M_1}\langle D_1\phi_1,\rho_1\rangle dx_1\nonumber\\
&=&
\varepsilon_{\mathcal{R},\alpha_1,\alpha_2}
\int_{M_1}\langle\phi_{1;aa}+E\phi_1+0\cdot\tau_1\phi_1,\rho_1\rangle dx_1
\label{eqn-4.b}\\
&=&\varepsilon_{\DR,\alpha_1,\alpha_2}
\int_{M_1}\left\{\vphantom{\int}-\langle\phi_{1;a},\rho_{1;a}\rangle
+\langle E\phi_1,\rho_1\rangle+\langle 0\tau\phi_1,\rho_1\rangle
\right\}dx_1\,.\nonumber
\end{eqnarray}
On the other hand, when we use the formulas of Lemma~\ref{lem-1.4} directly,
most of the terms vanish and we obtain:
\begin{eqnarray}
&&\beta_{2,\alpha_1,\alpha_2}^{\partial M}(\phi,\rho,D,\BB_\RR)\label{eqn-4.c}\\
&=&\displaystyle\int_{M_1}
\left\{\varepsilon_{\DR,\alpha_1,\alpha_2}^{12}\langle\phi_{1;a},\rho_{1;a}\rangle
+\langle\varepsilon_{\DR,\alpha_1,\alpha_2}^6E\phi_1
+\varepsilon_{\DR,\alpha_1,\alpha_2}^{13}\tau\phi_1,\rho_1\rangle\right\}dx_1\,.
\nonumber
\end{eqnarray}
The desired  result follows by comparing Equation~(\ref{eqn-4.b}) with
Equation~(\ref{eqn-4.c}).
\end{proof}

\subsection{Index shifting}
The following observation is simply change of notation:
\begin{lemma}\label{lem-4.3}
\ \begin{enumerate}
\item Suppose that $\phi^0=0$.
Then $\phi$ can also be regarded as being an element
$\tilde\phi\in\mathcal{K}_{\alpha_1-1}(V)$ where
$\tilde\phi^i=\phi^{i+1}$. We have:
$$\beta_{j-1,\alpha_1-1,\alpha_2}^{\partial M}(\tilde\phi,\rho,D,\BB_\DR)=
    \beta_{j,\alpha_1,\alpha_2}^{\partial M}(\phi,\rho,D,\BB_\DR)\,.$$
\item Suppose that $\rho^0=0$.
Then $\rho$ can also be regarded as being an element
$\tilde\rho\in\mathcal{K}_{\alpha_2-1}(V^*)$ where
$\tilde\rho^i=\rho^{i+1}$. We have:
$$\beta_{j-1,\alpha_1,\alpha_2-1}^{\partial M}(\phi,\tilde\rho,D,\BB_\DR)=
    \beta_{j,\alpha_1,\alpha_2}^{\partial M}(\phi,\rho,D,\BB_\DR)\,.$$
\item We have
$$\begin{array}{lll}
\varepsilon_{\DR,\alpha_1,\alpha_2}^1
=\varepsilon_{\DR,\alpha_1-1,\alpha_2},&
\varepsilon_{\DR,\alpha_1,\alpha_2}^3
=\varepsilon_{\DR,\alpha_1,\alpha_2-1},\
\vphantom{\vrule height 12pt}\\
\varepsilon_{\DR,\alpha_1,\alpha_2}^4
=\varepsilon_{\DR,\alpha_1-2,\alpha_2},&
\varepsilon_{\DR,\alpha_1,\alpha_2}^7
=\varepsilon_{\DR,\alpha_1,\alpha_2-2},
\vphantom{\vrule height 12pt}\\
 \varepsilon_{\DR,\alpha_1,\alpha_2}^5
 =\varepsilon_{\DR,\alpha_1-1,\alpha_2}^2,&
  \varepsilon_{\DR,\alpha_1,\alpha_2}^8
 =\varepsilon_{\DR,\alpha_1,\alpha_2-1}^2,\vphantom{\vrule height 12pt}\\
\varepsilon_{\DR,\alpha_1,\alpha_2}^{14}
=\varepsilon_{\DR,\alpha_1-1,\alpha_2-1},&
\varepsilon_{\RR,\alpha_1,\alpha_2}^{17}=
   \varepsilon_{\RR,\alpha_1-1,\alpha_2}^{15},
\vphantom{\vrule height 12pt}\\
\varepsilon_{\RR,\alpha_1,\alpha_2}^{18}
   =\varepsilon_{\RR,\alpha_1,\alpha_2-1}^{15}.
   \vphantom{\vrule height 12pt}
\end{array}$$
\end{enumerate}\end{lemma}

\begin{proof} Suppose $\phi^0=0$. Then
$$\phi\sim r^{-\alpha_1}(r\phi^1+r^2\phi^2+...)\sim r^{-(\alpha_1-1)}(\phi^1+r\phi^2+...).\,.$$
 Consequently, $\phi\in\mathcal{K}_{\alpha_1-1}$. We expand:
\begin{eqnarray*}
&&\beta(\phi,\rho,D,\BB_\DR)(t)\\
&\sim&
\sum_{n=0}^\infty t^n\beta_n^M(\phi,\rho,D)
   +\sum_{j=0}^\infty t^{(1+j-\alpha_1-\alpha_2)/2}
   \beta_{j,\alpha_1,\alpha_2}^{\partial
   M}(\phi,\rho,D,\BB_\DR)\\
&=&\beta(\tilde\phi,\rho,D,\BB_\DR)(t)\\
   &\sim&
\sum_{n=0}^\infty t^n\beta_n^M(\tilde\phi,\rho,D)
   +\sum_{k=0}^\infty t^{(1+k-(\alpha_1-1)-\alpha_2)/2}
   \beta_{k,\alpha_1-1,\alpha_2}^{\partial
   M}(\tilde\phi,\rho,D,\BB_\DR)\,.
\end{eqnarray*}
We set $k+1=j$ and equate terms in the boundary asymptotic expansions to prove Assertion~(1).
The proof of Assertion~(2) is similar and is therefore omitted.

We apply Assertion~(1) to prove Assertion~(3). Suppose $\phi^0=0$.
By Lemma~\ref{lem-1.4}:
\begin{eqnarray*}
&&\beta_{1,\alpha_1,\alpha_2}^{\partial M}(\phi,\rho,D,\BB_\DR)=
\int_{\partial M}\varepsilon_{\DR,\alpha_1,\alpha_2}^1\langle\phi^1,\rho^0\rangle
dy,\\
&&\beta_{0,\alpha_1-1,\alpha_2}^{\partial M}(\tilde\phi,\rho,D,\BB_\DR)=
\int_{\partial M}\varepsilon_{\DR,\alpha_1-1,\alpha_2}
\langle\tilde\phi^0,\rho^0\rangle dy\,.
\end{eqnarray*}
Consequently,
$\varepsilon_{\DR,\alpha_1,\alpha_2}^1
=\varepsilon_{\DR,\alpha_1-1,\alpha_2}$. The remaining assertions of the Lemma
are established similarly.
\end{proof}

\subsection{Relating Robin and Dirichlet Boundary Conditions}
We work on the interval $[0,\pi]$. Let $0\ne c\in\mathbb{R}$.
Set
$$
A:=\partial_x+c,\quad A^*:=-\partial_x+c,\quad
D=A^*A=AA^*=-(\partial_x^2-c^2)\,.
$$
Let $\BB_\DD$ be the Dirichlet boundary operator
and let $\BB_\RR\phi=0$ define the boundary condition
$A^*\phi|_{\partial M}=0$;
this is the associated Robin boundary operator:
$$\begin{array}{ll}
\BB_\DD\phi:=\phi(0)\oplus\phi(\pi),&
\BB_\RR\phi:=(\phi^\prime(0)-c\phi(0))\oplus
(-\phi^\prime(\pi)+c\phi(\pi)),\\
\SR(0)=-c,&\SR(\pi)=+c,\qquad\qquad E=-c^2\,.\vphantom{\vrule height 11pt}
\end{array}$$

\begin{lemma}\label{lem-4.4}
If $\Re(\alpha_1)<-1$ and $\Re(\alpha_2)<-1$, then:
\begin{enumerate}
\medbreak\item $\displaystyle\partial_t\beta(\phi,\rho,D,\BB_{\DR})(t)=
-\beta(D\phi,\rho,D,\BB_{\DR})(t)$.
\medbreak\item  $\partial_t\beta(\phi,\rho,D,\BB_{\DR})(t)=
-\beta(\phi,D\rho,D,\BB_{\DR})(t)$.
\medbreak\item $\displaystyle\partial_t
\beta(\phi,\rho,D,\BB_\RR)(t)=
-\beta(A^\star\phi,A^\star\rho,D,\BB_\DD)(t)$.
\medbreak\item $\displaystyle\partial_t
\beta(\phi,\rho,D,\BB_\DD)(t)=
-\beta(A\phi,A\rho,D,\BB_\RR)(t)$.
\medbreak\item $\beta^{\partial M}_{j,\alpha_1+2,\alpha_2}(D\phi,\rho,D,\BB_\DR)
=-\frac12(1+j-\alpha_1-\alpha_2)\beta_{j,\alpha_1,\alpha_2}^{\partial M}(\phi,\rho,D,
\BB_\DR)$.
\medbreak\item $\beta _{j,\alpha_1,\alpha_2+2}^{\partial M}(\phi,D\rho,D,\BB_\DR)
=-\frac12(1+j-\alpha_1-\alpha_2)\beta_{j,\alpha_1,\alpha_2}^{\partial M}
(\phi,\rho,D,\BB_\DR)$.
\medbreak\item $\beta_{j,\alpha_1+1,\alpha_2+1}^{\partial M}
(A^*\phi,A^*\rho,D,\BB_\DD)
=-\frac12(1+j-\alpha_1-\alpha_2)\beta_{j,\alpha_1,\alpha_2}^{\partial M}
(\phi,\rho,D,\BB_\RR)$.
\medbreak\item $\beta^{\partial M}_{j,\alpha_1+1,\alpha_2+1}
(A\phi,A\rho,D,\BB_\RR)
=-\frac12(1+j-\alpha_1-\alpha_2)\beta_{j,\alpha_1,\alpha_2}^{\partial M}
(\phi,\rho,D,\BB_\DD)$.
\medbreak\item $\varepsilon_{\RR,\alpha_1,\alpha_2}^{15}=\Xfifteen$.
\medbreak\item $\varepsilon_{\RR,\alpha_1,\alpha_2}^{17}=\Xseventeen$.
\medbreak\item $\varepsilon_{\RR,\alpha_1,\alpha_2}^{18}=\Xeighteen$.
\medbreak\item $
\varepsilon_{\RR,\alpha_1,\alpha_2}^{16}=\Xsixteen$
\end{enumerate}
\end{lemma}

\begin{proof}
We begin by establishing the spectral resolutions of the Dirichlet realization $D_\DD$ and of the Robin realization $D_\RR$.
For $n=1,2,...$, set:
$$
\phi^\DD_n:=\textstyle\sqrt\frac2\pi\sin(nx)\text{ and }\lambda_n:=n^2+c^2\,.
$$
Because $\{\phi_n^\DD\}$ is a complete
 orthonormal basis for $L^2$, because
$D\phi_n^\DD=\lambda_n\phi_n^\DD$, and because
$\phi_n^\DD(0)=\phi_n^\DD(\pi)=0$, we may conclude that
$\{\phi_n^\DD,\lambda_n\}_{n\in\mathbb{N}}$
is the Dirichlet spectral resolution of $D$. Similarly, set:
$$
\phi^\RR_n:=\lambda_n^{-\frac12}A\phi_n^\DD\,;
$$
$\{\phi^{\RR}_n\}$ is a complete
orthonormal basis for $L^2$ with
$D\phi_n^\RR=\lambda_n\phi_n^\RR$. Since
$$
A^*\phi_n^\RR|_{\partial M}
=\lambda_n^{-1/2}A^*A\phi_n^\DD|_{\partial M}
=\lambda_n^{-1/2}D\phi_n^\DD|_{\partial M}
=\lambda_n^{1/2}\phi_n^\DD|_{\partial M}
=0\,,
$$
$\BB_\RR A\phi_n^\DD=0$ so the eigenfunctions $\phi_n^\RR$
satisfy Robin boundary conditions. Consequently
$\{\phi_n^\RR,\lambda_n\}_{n\in\mathbb{N}}$
is a Robin spectral resolution of $D$.
If $\phi\in L^1$, let
$$\gamma_n^\DR(\phi):=\int_0^\pi\phi(x)\phi_n^\DR(x)dx$$
be the associated Fourier coefficients. We then have
\begin{eqnarray*}
&&\beta(\phi,\rho,D,\BB_{\DR})(t)
    =\sum_{n=1}^\infty e^{-t\lambda_n}\gamma_n^{\DR}(\phi)
    \gamma_n^{\DR}(\rho),\\
&&\partial_t\beta(\phi,\rho,D,\BB_{\DR})(t)
    =\sum_{n=1}^\infty -\lambda_ne^{-t\lambda_n}
    \gamma_n^{\DR}(\phi)\gamma_n^{\DR}(\rho)\,.
\end{eqnarray*}

Since $\Re(\alpha_1)<-1$ and $\Re(\alpha_2)<-1$,
we may integrate by parts to see
\begin{eqnarray*}
\gamma_n^{\DR}(D\phi)&=&
\int_MD\phi\cdot\phi_n^{\DR}dx=
\int_M\phi\cdot D\phi_n^{\DR}dx\\
&=&\lambda_n\int_M\phi\cdot \phi_n^{\DR}dx
=\lambda_n\gamma_n^{\DR}(\phi)\,.
\end{eqnarray*}
Assertions (1) and (2) now follow. To prove Assertion (3), we note:
\begin{eqnarray*}
&&\gamma_n^\DD(A^\star\phi)
=\int_MA^\star\phi\cdot\phi_n^\DD dx
=\int_M\phi\cdot A\phi_n^\DD dx\\
&&\qquad\qquad\phantom{}=\sqrt{\lambda_n}\int_M\phi\cdot\phi_n^\RR dx
=\sqrt{\lambda_n}\gamma_n^\RR(\phi)\,.
\end{eqnarray*}
Thus we may establish Assertion~(3) by computing:
\begin{eqnarray*}
&&\displaystyle\partial_t\beta(\phi,\rho,D,\BB_\RR)(t)=
\sum_{n=1}^\infty-\lambda_ne^{-t\lambda_n}\gamma_n^\RR(\phi)\cdot
\gamma_n^\RR(\rho)\\
&&\quad=-\sum_{n=1}^\infty e^{-t\lambda_n}\gamma_n^\DD(A^*\phi) \cdot
\gamma_n^\DD(A^*\rho)
=-\beta(A^*\phi,A^*\rho,D,\BB_\DD)(t)\,.
\end{eqnarray*}
Assertion~(4) follows similarly once we observe that
$\{A^*\phi_n^\RR/\sqrt{\lambda_n},\lambda_n\}_{n\in\mathbb{Z}}$ is
dually a spectral resolution of the Dirichlet Laplacian; the roles of $A$ and $A^*$
being interchanged.

We now establish Assertions~(5)-(8). Since
$D\phi\in\mathcal{K}_{\alpha_1+2}(V)$, we have:
\begin{eqnarray*}
&&\partial_t\beta(\phi,\rho,D,\BB_\DR)(t)\sim
\sum_{n=0}^\infty n\cdot(-1)^nt^{n-1}/n!
\mathcal{I}_{Reg}^g\{\langle D^n\phi,\rho\rangle\}\\
&&\qquad+\frac12\sum_{j=0}^\infty (j+1-\alpha_1-\alpha_2)
t^{(j-1-\alpha_1-\alpha_2)/2}\beta_{j,\alpha_1,\alpha_2}^{\partial
   M}(\phi,\rho,D,\BB_\DR)\\
&=&-\beta(D\phi,\rho,D,\BB_\DR)(t)\sim
-\sum_{\ell=0}^\infty(-1)^\ell t^\ell/\ell!\mathcal{I}_{\operatorname{Reg}}
\{\langle D^{\ell+1}\phi,\rho\rangle\}\\
&&\qquad-\sum_{k=0}^\infty
t^{(k+1-(\alpha_1+2)-\alpha_2)/2}\beta_{k,\alpha_1+2,\alpha_2}^{\partial
   M}(D\phi,\rho,D,\BB_\DR)\,.
\end{eqnarray*}
The term with $n=0$ plays no role. Setting $n=\ell+1$ equates the summation
involving the regularized integrals. Setting $k=j$ and equating terms in the asymptotic series for the boundary terms establishes
Assertion~(5); the proof of Assertion~(6) is similar. To prove Assertions (7) and (8),
we may integrate by parts to express
\begin{eqnarray*}
&&\mathcal{I}_{\operatorname{Reg}}\{\langle D^nA^*\phi,A^*\rho\rangle\}
=\mathcal{I}_{\operatorname{Reg}}\{\langle AD^nA^*\phi,\rho\rangle\}
=\mathcal{I}_{\operatorname{Reg}}\{\langle D^nAA^*\phi,\rho\rangle\}\\
&=&\mathcal{I}_{\operatorname{Reg}}\{\langle D^{n+1}\phi,\rho\rangle\},\\
&&\mathcal{I}_{\operatorname{Reg}}\{\langle D^nA\phi,A\rho\rangle\}
=\mathcal{I}_{\operatorname{Reg}}\{\langle A^*D^nA\phi,\rho\rangle\}
=\mathcal{I}_{\operatorname{Reg}}\{\langle D^nA^*A\phi,\rho\rangle\}\\
&=&\mathcal{I}_{\operatorname{Reg}}\{\langle D^{n+1}\phi,\rho\rangle\}\,.
\end{eqnarray*}
The remainder of the argument now follows similarly.

To prove Assertion~(9), we assume $\Re(\alpha_1)<-1$ and $\Re(\alpha_2)<-1$.
We then use analytic continuation to obtain the result on the full parameter range.
Adopt the notation of Equation~(\ref{eqn-3.a}) and set
$\phi=\phi_{\alpha_1}$, and $\rho=\rho_{\alpha_2}$. Only the value at $r=0$
is relevant; integration over the boundary is just evaluation at $0$.
 Since $D=-(\partial_x^2-c^2)$, the metric and associated connection are flat. We have:
\begin{eqnarray*}
&&\ A^*\phi=(-\partial_x+c)(r^{-\alpha_1})=r^{-\alpha_1-1}\{\alpha_1+cr\},\\
&&\ A^*\rho=(-\partial_x+c)(r^{-\alpha_2})=r^{-\alpha_2-1}\{\alpha_2+cr\},\\
&&\begin{array}{llll}
\phi^0=1,&\phi^1=0,&\rho^0=1,&\rho^1=0,\\
(A^*\phi)^0=\alpha_1,&(A^*\phi)^1=c,
&(A^*\rho)^0=\alpha_2,&(A^*\rho)^1=c,\\
\phi^2=0,&\rho^2=0,&(A^*\rho)^2=0,&(A^*\phi)^2=0,\\
 \SR(0)=-c,&E=-c^2\,.
\end{array}\end{eqnarray*}
 We use Assertion~(7),  Lemma~\ref{lem-1.4}, and Lemma~\ref{lem-4.3}~(3) to see:
\begin{eqnarray*}
&&\beta^{\partial M}_{1,\alpha_1+1,\alpha_2+1}
(A^*\phi,A^*\rho,D,\mathcal{B}_\DD)\\
&=&
\varepsilon_{\mathcal{D},\alpha_1+1,\alpha_2+1}^1(A^*\phi)^1\cdot(A^*\rho)^0
+\varepsilon_{\mathcal{D},\alpha_1+1,\alpha_2+1}^3(A^*\phi)^0\cdot(A^*\rho)^1\\
&=&c\cdot
\left\{\alpha_2\varepsilon_{\mathcal{D},\alpha_1,\alpha_2+1}
+\alpha_1\varepsilon_{\mathcal{D},\alpha_1+1,\alpha_2}\right\}\\
&=&-\textstyle\frac12(2-\alpha_1-\alpha_2)\beta^{\partial M}_{1,\alpha_1,\alpha_2} 
(\phi,\rho,D,\mathcal{B}_\RR)\\
&=&\textstyle\frac12(2-\alpha_1-\alpha_2)\cdot
c\varepsilon_{\RR,\alpha_1,\alpha_2}^{15}\,.
\end{eqnarray*}
We solve these equations to establish see:
$$\textstyle\varepsilon_{\RR,\alpha_1,\alpha_2}^{15}
=\textstyle\frac2{2-\alpha_1-\alpha_2}(\alpha_2\varepsilon_{\mathcal{D},\alpha_1,\alpha_2+1}
     +\alpha_1\varepsilon_{\mathcal{D},\alpha_1+1,\alpha_2})\,.$$
Applying Lemma~\ref{lem-1.7}~(3) then yields:
$$\textstyle\varepsilon_{\RR,\alpha_1,\alpha_2}^{15}=\Xfifteen$$
Assertion~(10) and Assertion~(11) follow from Assertion~(9) by
using Lemma~\ref{lem-4.3} to index shift. We take $j=2$ and examine the
coefficient of $c^2$ to complete the proof of Assertion~(12) by computing:
\medbreak\quad
$\beta^{\partial M}_{2,\alpha_1+1,\alpha_2+1}
(A^*\phi,A^*\rho,D,\mathcal{B}_\DD)$
\medbreak\qquad
$=(A^*\phi)^1(A^*\rho)^1\varepsilon_{\DD,\alpha_1+1,\alpha_2+1}^{14}
-c^2 (A^*\phi)^0(A^*\rho)^0\varepsilon_{\DD,\alpha_1+1,\alpha_2+1}^{6}$
\medbreak\qquad
$=c^2\cdot\varepsilon_{\mathcal{D},\alpha_1,\alpha_2}
-c^2\alpha_1\alpha_2\varepsilon_{\DD,\alpha_1+1,\alpha_2+1}$
\medbreak\qquad
$=-\textstyle\frac12(3-\alpha_1-\alpha_2)
\beta^{\partial M} _{2,\alpha_1,\alpha_2}(\phi,\rho,D,\mathcal{B}_\RR)$
\medbreak\qquad
$=-\textstyle\frac12(3-\alpha_1-\alpha_2)c^2\{\varepsilon_{\RR,\alpha_1,\alpha_2}^{16}
-\varepsilon_{\RR,\alpha_1,\alpha_2}\}$.
\medbreak\noindent We solve these equations and use Lemma~\ref{lem-1.7} to see:
\medbreak\qquad
$\varepsilon_{\RR,\alpha_1,\alpha_2}^{16}=\textstyle
-\frac2{3-\alpha_1-\alpha_2}\varepsilon_{\DD,\alpha_1,\alpha_2}+\frac{2\alpha_1\alpha_2}{3-\alpha_1-\alpha_2}\varepsilon_{\DD,\alpha_1+1,\alpha_2+1}
+\varepsilon_{\RR,\alpha_1,\alpha_2}$
\medbreak\qquad\qquad
$=\textstyle-\frac2{3-\alpha_1-\alpha_2}\cdot\frac{3-\alpha_1-\alpha_2}{-2(\alpha_1-1)(\alpha_2-1)}
\varepsilon_{\mathcal{R},\alpha_1-1,\alpha_2-1}$
\medbreak\qquad\qquad\quad
$+\textstyle\frac{2\alpha_1\alpha_2}{3-\alpha_1-\alpha_2}\frac{1-\alpha_1-\alpha_2}{-2\alpha_1\alpha_2}
\varepsilon_{\mathcal{R},\alpha_1,\alpha_2}
+\frac{3-\alpha_1-\alpha_2}{3-\alpha_1-\alpha_2}\varepsilon_{\RR,\alpha_1,\alpha_2}$
\medbreak\qquad\qquad$=\textstyle\Xsixteen$.\end{proof}

\subsection{The proof of Lemma~\ref{lem-1.7}}
Lemma~\ref{lem-4.3}~(3) focuses attention on the shifted invariants
$\varepsilon_{\BB,\alpha_1-2,\alpha_2}$,
$\varepsilon_{\BB,\alpha_1,\alpha_2-2}$, and
$\varepsilon_{\BB,\alpha_1-1,\alpha_2-1}$.
Lemma \ref{lem-1.7} expresses these invariants in terms of the fundamental
invariants $\varepsilon_{\BB,\alpha_1,\alpha_2}$. These relations follow from
the standard properties of the $\Gamma$ function. It is instructive, however, to
use Lemma~\ref{lem-4.4} to establish these properties as an independent check
on our work.

We adopt the notation of Lemma~\ref{lem-4.3} with $c=0$. We use Equation~(\ref{eqn-3.a}) to define
$\phi_{\alpha_1}$ and $\rho_{\alpha_2}$ on $[0,\pi]$.
We take $j=0$ and apply Lemma~\ref{lem-1.4} and Lemma~\ref{lem-4.4} to see:
\begin{eqnarray*}
&&\beta_{0,\tilde\alpha_1+2,\tilde\alpha_2}^{\partial M}
(D\phi_{\tilde\alpha_1},\rho_{\tilde\alpha_2},D,\BB_\DR)\\
&=&
-\beta_{0,\tilde\alpha_1+2,\tilde\alpha_2}^{\partial M}(\tilde\alpha_1(\tilde\alpha_1+1)\phi_{\tilde\alpha_1+  2},
\rho_{\tilde\alpha_2} ,D,\BB_\DR)
=-\tilde\alpha_1(\tilde\alpha_1+1)
\varepsilon_{\DR,\tilde\alpha_1+2,\tilde\alpha_2}\\
&=&-\textstyle\frac{1-\tilde\alpha_1-\tilde\alpha_2}2
\beta_{0,\tilde\alpha_1,\tilde\alpha_2}^{\partial M}
(\phi,\rho,D,\BB_\DR)
=-\textstyle\frac{1-\tilde\alpha_1-\tilde\alpha_2}2
\varepsilon_{\DR,\tilde\alpha_1,\tilde\alpha_2}\,.
\end{eqnarray*}
Consequently:
$$
-\tilde\alpha_1(\tilde\alpha_1+1)
\varepsilon_{\DR,\tilde\alpha_1+2,\tilde\alpha_2}
=-\textstyle\frac{1-\tilde\alpha_1-\tilde\alpha_2}2
\varepsilon_{\DR,\tilde\alpha_1,\tilde\alpha_2}\,.
$$
Lemma~\ref{lem-1.7}~(1) now follows after replacing setting $\tilde\alpha_1=\alpha_1-2$ and
$\tilde\alpha_2=\alpha_2$.
The proof of Lemma~\ref{lem-1.7}~(2) is similar.
To prove Lemma~\ref{lem-1.7}~(3), we set $c=0$ to obtain $A=-A^*$; the roles of $\mathcal{D}$ and $\mathcal{R}$ are
then symmetric in Lemma~\ref{lem-4.4} so we may compute:
\begin{eqnarray*}
&&\beta_{0,\tilde\alpha_1+1,\tilde\alpha_2+1}^{\partial M}
(\partial_x\phi_{\tilde\alpha_1},\partial_x\rho_{\tilde\alpha_2},D,\BB_\DR)\\
&=&\tilde\alpha_1\tilde\alpha_2
\beta_{0,\tilde\alpha_1+1,\tilde\alpha_1+1}^{\partial M}(\phi_{\tilde\alpha_1+1},
\rho_{\tilde\alpha_2+1},\BB_\DR)
=\tilde\alpha_1\tilde\alpha_2\varepsilon_{\DR,\tilde\alpha_1+1,\tilde\alpha_2+1}\\
&=&-\textstyle\frac{1-\tilde\alpha_1-\tilde\alpha_2}2
\beta_{0,\tilde\alpha_1,\tilde\alpha_2}^{\partial M}
(\phi,\rho,D,\BB_{\mathcal{R}/\mathcal{D}})
=-\textstyle\frac{1-\tilde\alpha_1-\tilde\alpha_2}2
\varepsilon_{{\mathcal{R}/\mathcal{D}},
\tilde\alpha_1,\tilde\alpha_2}\end{eqnarray*}
Lemma~\ref{lem-1.7}~(3) now follows after setting $\tilde\alpha_i=\alpha_i-1$.\hfill\qed

\subsection{Warped products}\label{sect-4.5}
Let $\mathbb{T}^{m-1}=[0,2\pi]^{m-1}$ denote the torus
where we identify $0\sim2\pi$ to form a closed manifold. Let
$(y_1,...,y_{m-1})$ be the usual periodic parameters where
$0\le y_i\le 2\pi$, and let
$$ M:=\mathbb{T}^{m-1}\times[0,\pi]\,.$$
Let
$f_a\in C^\infty([0,\pi])$ be a collection of smooth functions which have compact support near $r=0$ with
$f_a(0)=0$. Let  $\chi(r)\in C^\infty([0,\pi])$ be the mesa function
described previously; $\chi$ is identically 1 near $r=0$ and identically 0 near $r=1$. Set
$$\begin{array}{ll}
ds^2_M=\displaystyle\sum_ae^{2f_a(r)}dy_a\circ dy_a+dr\circ dr,&
\phi_2:=\chi(r)r^{-\alpha_1},\\
D_M:=-\displaystyle\sum_a
        e^{-2f_a(r)}\partial_{y_a}^2-\partial_r^2,&
\rho_2:=\chi(r)r^{-\alpha_2}\vphantom{\vrule height 14pt},\\
D_2:=-\partial_r^2,\quad \phi_M(y,r):=\phi_2(r),&\rho_M(y,r)=e^{-\sum_af_a(r)}\rho_2(r).
\end{array}$$
Let $\BB=\BB_\DD$ or let $\BB_\RR=\partial_r+S_{\mathcal{R},0}$.

\begin{lemma}\label{lem-4.5}
Adopt the notation established above. Then:
$$\beta_{j,\alpha_1,\alpha_2}(\phi_M,\rho_M,D,\mathcal{B})(y,0)=
\beta_{j,\alpha_1,\alpha_2}(\phi_2,\rho_2,D_2,\mathcal{B})(0)\,.$$
\end{lemma}

\begin{proof} Let
$T_2:=e^{-tD_{2,\mathcal{B}}}\phi_2$ and let $T_M:=e^{-tD_{M,\BB}}\phi_M$. We show that
$$
T_M(y,r;t)=T_2(r;t)
$$
 by
verifying that $T$ satisfies the defining relations
of Equation~(\ref{eqn-1.c}):
\begin{eqnarray*}
&&(\partial_t+D_M)T_2=(\partial_t+D_2)T_2=0,\quad
\lim_{t\downarrow0}T_2(\cdot;t)=\phi_2=\phi_{M},\quad
\mathcal{B}T_2(\cdot;t)=0\,.
\end{eqnarray*}
Since the volume element on $M$ is given by $e^{\sum_af_a(r)}dydr$, we
complete the proof by computing:
\medbreak\qquad
$\beta(\phi_M,\rho_M,D_M,\mathcal{B})(t)$
\medbreak\qquad\quad
$=\int_{\mathbb{T}^{m-1}\times[0,1]}T_2(r;t)
     \{e^{-\sum_af_a(r)}\rho_2(r)\}e^{\sum_af_a(r)}dydr$
\medbreak\qquad\quad
$=\int_{\mathbb{T}^{m-1}\times[0,1]}T_2(r;t)\rho_2(r)dydr
=(2\pi)^{m-1}\int_{[0,1]}T_2(r;t)\rho_2(r)dr$
\medbreak\qquad\quad
$=(2\pi)^{m-1}\beta(\phi_2,\rho_2,D_2,\mathcal{B})(t)$.
\end{proof}

\subsection{The proof of Theorem~\ref{thm-1.9}}
We have computed most of the unknown coefficients in Lemma~\ref{lem-1.6}.
We use adopt the notation of Section~\ref{sect-4.5} and apply Lemma~\ref{lem-4.5} to
compute the remaining coefficients; the recursion relations
of Lemma~\ref{lem-1.7} also play a crucial role as do our previous computations.
We adopt the notation of Equation~(\ref{eqn-3.a}) and consider the following example.
Let $\phi_2:=r^{-\alpha_1}\chi$ and $\rho_2:=r^{-\alpha_2}\chi$ where
$\chi$ is identically $1$ near $r=0$ and vanishes for $r\ge\frac12$. We then have:
$$\phi_2^0=\rho_2^0=1,\quad\phi_2^\ell=\rho_2^\ell=0\text{ for }\ell\ge1\,.$$
 Consequently there exist suitably chosen constants $\{c_j\}$ (which play no role in our development)
so that:
\begin{eqnarray*}
\beta_{j,\alpha_1,\alpha_2}(\phi_M,\rho_M,D_M,\mathcal{B})(y,0)&=&
\beta_{j,\alpha_1,\alpha_2}(\phi_2,\rho_2,D_2,\mathcal{B})(0)\\
&=&c_j\cdot S_{\mathcal{R},0}^j
\quad\text{for}\quad j>0\,.
\end{eqnarray*}
We shall ignore this term and concentrate on other terms.
Let $\omega$ be the connection 1-form
of the connection on $V$ defined by $D$ and let $\tilde\omega$ be the dual
connection 1-form of the dual connection on $V^*$; $\tilde\omega+\omega=0$.
We let $e_a:=\partial_{y_a}$ for $1\le a\le m$ and $e_m:=\partial_r$; this is
an orthonormal frame on $\partial M$. We use Lemma~\ref{lem-1.1} to compute:
$$
\begin{array}{ll}
\Gamma_{abm}=-f_a^\prime\delta_{ab}e^{2f_a},&\Gamma_{ab}{}^m=-f_a^\prime e^{2f_a}\delta_{ab},\\
\Gamma_{amb}=f_a^\prime\delta_{ab}e^{2f_a}
,&\Gamma_{am}{}^b=f_a^\prime\delta_a^b,
   \vphantom{\vrule height 12pt}\\
\left.L_{ab}\right.|_{\partial M}
=\left.\Gamma_{ab}{}^m\right|_{\partial M}
=\left.-f_a^\prime\delta_{ab}\right|_{\partial M},&\vphantom{\vrule height 12pt}\\
\textstyle\omega_a=0,&\tilde\omega_a=0,
   \vphantom{\vrule height 12pt}\\
\omega_m=-{\textstyle\frac12}\sum_af_a^\prime,&
\tilde\omega_m=-\omega_m={\textstyle\frac12}\sum_af_a^\prime\,.\vphantom{\vrule
height 12pt}
\end{array}$$
Consequently: \medbreak\qquad
$\left.R_{ambm}\right|_{\partial M}=\left.g((\nabla_a\nabla_m-\nabla_m\nabla_a)e_b,e_m)\right|_{\partial M}
=\left.\left\{\Gamma_{ac}{}^m\Gamma_{mb}{}^c-\partial_m\Gamma_{ab}{}^m\right\}\right|_{\partial M}$
\medbreak\qquad\quad
$=\left.\left\{-(f_a^\prime)^2+f_a^{\prime\prime}+2(f_a^\prime)^2\}\delta_{ab}\right\}\right|_{\partial M}$,
\medbreak\qquad
$\left.\operatorname{Ric}_{mm}\right|_{\partial M}
=\left.-\textstyle\sum_a\left\{f_a^{\prime\prime}+(f_a^\prime)^2\right\}\right|_{\partial M}$,
\medbreak\qquad $E|_{\partial
M}=\left.\left\{\vphantom{\textstyle\frac12\sum_a}-\partial_m\omega_m
-\omega_m^2+\omega_m\Gamma_{aa}{}^m\right\}\right|_{\partial M}$
\medbreak\qquad\quad
$=\left.\left\{\textstyle\frac12\sum_af_a^{\prime\prime}-\frac14\sum_{a,b}f_a^\prime
f_b^\prime+\frac12\sum_{a,b}f_a^\prime f_b^\prime\right\}\right|_{\partial M}$
\medbreak\qquad\quad
$=\left.\left\{\textstyle\frac12\sum_af_a^{\prime\prime}+\frac14\sum_{a,b}f_a^\prime
f_b^\prime\right\}\right|_{\partial M}$.
\medbreak\noindent At this stage, the connection $1$-form enters in a crucial fashion
as we must use the dual connection and the dual connection $1$-form $\tilde\omega$
to covariantly differentiate $\rho$. Thus $\nabla_{\partial r}\psi=(\partial_r\psi+\omega_m\psi)$ if
$\psi\in C^\infty(V)$ while $\nabla_{\partial r}\psi=(\partial_r\psi-\omega_m\psi)$ if $\psi\in V^*$. We
use Equation~(\ref{eqn-1.b}) to compute:
\medbreak\quad
$\phi_M^0=1$,
\medbreak\quad
$\phi_M^1=\left.\left\{\nabla_{\partial r}(r^{\alpha_1}\phi_M)\right\}\right|_{\partial M}
=\left.\left\{(\partial_r-\textstyle{\textstyle\frac12}\sum_af_a^\prime)(1)\right\}\right|_{\partial M}=
\left.-{\textstyle\frac12}\sum_af_a^\prime\right|_{\partial M}$,
\medbreak\quad
$\phi_M^2={\textstyle\frac12}\{(\nabla_{\partial_r})^2(r^{\alpha_1}\phi_M)\}|_{\partial M}
=\left.{\textstyle\frac12}\left\{(\partial_r-\textstyle{\textstyle\frac12}\sum_af_a^\prime)^2(1)\right\}\right|_{\partial M}$
\smallbreak\qquad\phantom{..}
$=\left.{\textstyle\frac18}\left\{\sum_{a,b}f_a^\prime f_b^\prime
-{\textstyle\frac14}\sum_af_a^{\prime\prime}\right\}\right|_{\partial M}$,
\medbreak\quad
$\rho_M^0=\left.e^{-\sum_af_a}\right|_{\partial M}=1$,
\medbreak\quad
$\rho_M^1=\left.\left\{\nabla^*_{\partial r}(\rho_M)\right\}\right|_{\partial M}
=\left.\left\{(\partial_r+{\textstyle\frac12}\sum_af_a^\prime)(e^{-\sum_af_a})\right\}\right|_{\partial
M}=-\left.\left\{{\textstyle\frac12}\sum_af_a^\prime\right\}\right|_{\partial M}$,
\medbreak\quad
$\rho_M^2=\left.{\textstyle\frac12}\left\{(\nabla^*_{\partial r})^2\rho_M\right\}\right|_{\partial M}
=\left.{\textstyle\frac12}\left\{(\partial_r+{\textstyle\frac12}\sum_af_a^\prime)^2
(e^{-\sum_af_a})\right\}\right|_{\partial M}$
\smallbreak\qquad\phantom{.}
$=\left.\textstyle{\textstyle\frac18}\left\{\sum_{a,b}f_a^\prime f_b^\prime
-{\textstyle\frac14}\sum_af_a^{\prime\prime}\right\}\right|_{\partial M}$.
\medbreak\noindent
 To ensure that the Robin boundary on $M$ takes the form
$\mathcal{B}=\partial_r+S_{\mathcal{R},0}$, we set
$$\SR=S_{\mathcal{R},0}-\left.\omega_m\right|_{\partial M}
=S_{\mathcal{R},0}+\left.\left\{\textstyle\frac12\sum_af^\prime_a\right\}\right|_{\partial M}\,.$$

\subsubsection{The coefficient of $\sum_af_a^\prime$ in
$\beta_1^{\partial M}$}\
\medbreak\quad
$0=-{\textstyle\frac12}\varepsilon_{\RR,\alpha_1,\alpha_2}^1
-\varepsilon_{\RR,\alpha_1,\alpha_2}^2
-{{\textstyle\frac12}\varepsilon_{\RR,\alpha_1,\alpha_2}^3}
+\textstyle\frac12\varepsilon_{\RR,\alpha_1,\alpha_2}^{15}$.
\medbreak\noindent
Consequently
\medbreak\quad
$\varepsilon_{\mathcal{R},\alpha_1,\alpha_2}^2=
-{\textstyle\frac12}\varepsilon_{\RR,\alpha_1,\alpha_2}^1
-{\textstyle\frac12}\varepsilon_{\RR,\alpha_1,\alpha_2}^3
+{\textstyle\frac12}\varepsilon_{\RR,\alpha_1,\alpha_2}^{15}$
\medbreak\qquad
$=\textstyle\frac12\{-\Xone-\Xthree+\Xfifteen\}$
\medbreak\qquad
$=\Xtwo$
\medbreak\noindent
We use Lemma~\ref{lem-4.3} to shift indices and compute:
\medbreak\quad
$\varepsilon_{\mathcal{R},\alpha_1,\alpha_2}^5=
\varepsilon_{\mathcal{R},\alpha_1-1,\alpha_2}^2$
\medbreak\qquad
$=\Xfive$,
\medbreak\qquad
$\varepsilon_{\mathcal{R},\alpha_1,\alpha_2}^8=\varepsilon_{\mathcal{R},\alpha_1,\alpha_2-1}^2$
\medbreak\qquad
$=\Xeight$.

\subsubsection{The coefficient of $S_{\mathcal{R},0}\sum_af_a^\prime$ in $\beta_2^{\partial M}$}
$$\textstyle0=\varepsilon_{\RR,\alpha_1,\alpha_2}^{16}
-\frac12\varepsilon_{\RR,\alpha_1,\alpha_2}^{17}
-\frac12\varepsilon_{\RR,\alpha_1,\alpha_2}^{18}
-\varepsilon_{\RR,\alpha_1,\alpha_2}^{19}\,.
$$
Consequently
\medbreak\quad$\varepsilon_{\RR,\alpha_1,\alpha_2}^{19}
=\varepsilon_{\RR,\alpha_1,\alpha_2}^{16}-\frac12\varepsilon_{\RR,\alpha_1,\alpha_2}^{17}
-\frac12\varepsilon_{\RR,\alpha_1,\alpha_2}^{18}$
\medbreak\qquad$=
\textstyle\Xsixteen$
\smallbreak\qquad\quad$-\textstyle\frac12\Xseventeen$
\smallbreak\qquad\quad$-\textstyle\frac12\Xeighteen$
\medbreak\qquad$=
\textstyle\frac12\frac2{3-\alpha_1-\alpha_2}
\left\{2+\frac{(\alpha_1-2)(\alpha_1-1)}{\alpha_1-2}+\frac{(\alpha_2-2)(\alpha_2-1)}{\alpha_2-2}\right\}
\varepsilon_{\RR,\alpha_1,\alpha_2}$
\medbreak\qquad\quad$
+\textstyle\frac12\frac{1}
{(\alpha_1-1)(\alpha_2-1)}\{2+(\alpha_1-1)+(\alpha_2-1)\}\varepsilon_{\RR,\alpha_1-1,\alpha_2-1}$
\medbreak\qquad
$=\Xnineteen$.
\subsubsection{The coefficients of $\sum_af_a^{\prime\prime}$ and $\sum_a(f_a^{\prime})^2$ in $\beta_2^{\partial M}$}
We obtain the relations:
\begin{eqnarray*}
&&0=-{\textstyle\frac14}\varepsilon_{\RR,\alpha_1,\alpha_2}^4
+{\textstyle\frac12}\varepsilon_{\RR,\alpha_1,\alpha_2}^6
-{\textstyle\frac14}\varepsilon_{\RR,\alpha_1,\alpha_2}^7
-\varepsilon_{\RR,\alpha_1,\alpha_2}^9,\text{ and}\\
&&0=
\textstyle-\varepsilon_{\RR,\alpha_1,\alpha_2}^9
+\varepsilon_{\RR,\alpha_1,\alpha_2}^{11}\,.
\end{eqnarray*}
This shows
\medbreak\quad
$\varepsilon_{\RR,\alpha_1,\alpha_2}^9=\varepsilon_{\RR,\alpha_1,\alpha_2}^{11}
=-\frac14\varepsilon_{\RR,\alpha_1,\alpha_2}^4+\frac12\varepsilon_{\RR,\alpha_1,\alpha_2}^6
-\frac14\varepsilon_{\RR,\alpha_1,\alpha_2}^7$
\medbreak\qquad
$=\textstyle-\frac14\Xfour-\frac14\Xseven+\frac12\Xsix$
\medbreak\qquad
$=\left\{-\frac14\frac{2(\alpha_1-2)(\alpha_1-1)}{3-\alpha_1-\alpha_2}
-\frac14\frac{2(\alpha_2-2)(\alpha_2-1)}{3-\alpha_1-\alpha_2}+\frac12\right\}
\varepsilon_{\RR,\alpha_1,\alpha_2}$
\medbreak\qquad
$=\Xnine$.

\subsubsection{The coefficient of $\sum_{a,b}f_a^\prime f_b^\prime$ in
$\beta_2^{\partial M}$}
\begin{eqnarray*}
&&0=\textstyle
\frac18\varepsilon_{\RR,\alpha_1,\alpha_2}^4
+{\textstyle\frac12}\varepsilon_{\RR,\alpha_1,\alpha_2}^5
+{\textstyle\frac14}\varepsilon_{\RR,\alpha_1,\alpha_2}^6
+{\textstyle\frac18}\varepsilon_{\RR,\alpha_1,\alpha_2}^7
+{\textstyle\frac12}\varepsilon_{\RR,\alpha_1,\alpha_2}^{8}
+\varepsilon_{\RR,\alpha_1,\alpha_2}^{10}\\
&&\quad\textstyle+\frac14\varepsilon_{\RR,\alpha_1,\alpha_2}^{14}
+\frac14\varepsilon_{\RR,\alpha_1,\alpha_2}^{16}
-\frac14\varepsilon_{\RR,\alpha_1,\alpha_2}^{17}
-\frac14\varepsilon_{\RR,\alpha_1,\alpha_1}^{18}
-\frac12\varepsilon_{\RR,\alpha_1,\alpha_2}^{19}\,.
\end{eqnarray*}
This implies
\medbreak\quad
$\varepsilon_{\RR,\alpha_1,\alpha_2}^{10}=\textstyle-\frac18\Xfour-\frac14\Xsix-\frac18\Xseven-\frac14\Xfourteen$
      \medbreak\qquad$-\frac12\{\Xfive\}$\hfill($-\frac12\varepsilon_{\RR,\alpha_1,\alpha_2}^5$)
      \medbreak\qquad$-\frac12\{\Xeight\}$\hfill($-\frac12\varepsilon_{\RR,\alpha_1,\alpha_2}^8$)
      \medbreak\qquad$-\frac14\{\Xsixteen\}$\hfill($-\frac14\varepsilon_{\RR,\alpha_1,\alpha_2}^{16}$)
     \medbreak\qquad$+\frac14\Xseventeen$\hfill($\frac14\varepsilon_{\RR,\alpha_1,\alpha_2}^{17}$)
     \medbreak\qquad$+\frac14\Xeighteen$\hfill($\frac14\varepsilon_{\RR,\alpha_1,\alpha_2}^{18}$)
     \medbreak\qquad$+\frac12\{\Xnineteen\}$\hfill($\frac12\varepsilon_{\RR,\alpha_1,\alpha_2}^{19}$)
\medbreak\quad
$=\left\{-\frac18+\frac14\frac{\alpha_1-1}{\alpha_1-2}-\frac14\frac1{\alpha_1-2}\right\}\varepsilon_{\RR,\alpha_1-2,\alpha_2}
+\left\{-\frac18+\frac14\frac{\alpha_2-1}{\alpha_2-2}-\frac14\frac1{\alpha_2-2}
\right\}\varepsilon_{\RR,\alpha_1,\alpha_2-2}$
\medbreak\qquad
$+\left\{-\frac14+\frac14\frac{\alpha_2}{\alpha_2-1}+\frac14\frac{\alpha_1}{\alpha_1-1}
-\frac14\frac1{(\alpha_1-1)(\alpha_2-1)}-\frac14\frac1{\alpha_2-1}-\frac14\frac1{\alpha_1-1}\right.$
\medbreak\qquad\qquad$\left.
+\frac14\frac{\alpha_1+\alpha_2}
{(\alpha_1-1)(\alpha_2-1)}\right\}\varepsilon_{\RR,\alpha_1-1,\alpha_2-1}$
\medbreak\qquad\quad
$\left\{-\frac14-\frac14\frac2{3-\alpha_1-\alpha_2}+\frac12\frac{\alpha_1+\alpha_2}{3-\alpha_1-\alpha_2}\right\}
\varepsilon_{\RR,\alpha_1,\alpha_2}$
\medbreak\quad
$=\frac18\varepsilon_{\RR,\alpha_1-2,\alpha_2}
   +\frac18\varepsilon_{\RR,\alpha_1,\alpha_2-2}
   +\frac{3\alpha_1+3\alpha_2-5}{4(3-\alpha_1-\alpha_2)}\varepsilon_{\RR,\alpha_1,\alpha_2}$
\medbreak\qquad
$+\frac14\left\{-1+1+1+{\frac{\alpha_1+\alpha_2-1}
{(\alpha_1-1)(\alpha_2-1)}}\right\}\varepsilon_{\RR,\alpha_1-1,\alpha_2-1}$
\medbreak\quad
$=\frac{(\alpha_1-2)(\alpha_1-1)+(\alpha_2-2)(\alpha_1-1)+3\alpha_1+3\alpha_2{-5}}{4(3-\alpha_1-\alpha_2)}\varepsilon_{\RR,\alpha_1,\alpha_2}
+\frac14\frac{\alpha_1\alpha_2}{(\alpha_1-1)(\alpha_2-1)}\varepsilon_{\RR,\alpha_1-1,\alpha_2-1}$
\medbreak\quad
$=\Xten$.
\medbreak\noindent We have determined all the unknown coefficients in Lemma~\ref{lem-1.4}. This
completes the proof of Theorem~\ref{thm-1.9}. \hfill\qed

\section*{Acknowledgments}
Research of P. Gilkey partially supported by project MTM2009-07756
(Spain), by project 174012 (Serbia), and by the Korean Institute for Advanced Studies.
Research of H. Kang supported by a research fellowship from the Korea Institute for Advanced Study.

\end{document}